\documentclass[12pt]{amsart}
\usepackage{amssymb, hyperref,xcolor, enumerate, longtable, graphicx, caption, subcaption}

\usepackage[margin=1.3in]{geometry}

\newtheorem{theoremIntro}{Theorem}[]

\newtheorem{questionIntro}{Question}

\newtheorem{theorem}{Theorem}[section]
\newtheorem{lemma}[theorem]{Lemma}
\newtheorem{proposition}[theorem]{Proposition}

\newtheorem{definition}[theorem]{Definition}
\newtheorem{corollary}[theorem]{Corollary}

\theoremstyle{remark}
\newtheorem{remark}[theorem]{Remark}
\newtheorem{example}[theorem]{Example}

\numberwithin{equation}{section}

\newcommand{\calB}{\ensuremath{\mathcal{B}}}
\newcommand{\calG}{\ensuremath{\mathcal{G}}}
\newcommand{\calF}{\ensuremath{\mathcal{F}}}
\newcommand{\calL}{\ensuremath{\mathcal{L}}}

\newcommand{\xra}{\xrightarrow}

\newcommand{\mo}{{-1}}

\newcommand{\bbZ}{\ensuremath{\mathbb{Z}}}
\newcommand{\bbR}{\ensuremath{\mathbb{R}}}
\newcommand{\bbN}{\ensuremath{\mathbb{N}}}

\newcommand{\rk}{\ensuremath{\mathrm{rk}}}
\begin{document}

\vspace*{-5mm}
\title{Minimal additive complements in finitely generated abelian groups}

\author{Arindam Biswas}
\address{Universit\"at Wien, Fakult\"at f\"ur Mathematik, Oskar-Morgenstern-Platz 1, 1090 Wien, Austria \& Erwin Schr\"odinger International Institute for Mathematics and Physics (E.S.I.) Boltzmanngasse 9, 1090 Wien, Austria}
\curraddr{}
\email{arindam.biswas@univie.ac.at}
\thanks{}

\author{Jyoti Prakash Saha}
\address{Department of Mathematics, Indian Institute of Science Education and Research Bhopal, Bhopal Bypass Road, Bhauri, Bhopal 462066, Madhya Pradesh,
India}
\curraddr{}
\email{jpsaha@iiserb.ac.in}
\thanks{}



\begin{abstract}
Given two non-empty subsets $W,W'\subseteq G$ in an arbitrary abelian group $G$, $W'$ is said to be an additive complement to $W$ if $W + W'=G$ and it is minimal if no proper subset of $W'$ is a complement to $W$. The notion was introduced by Nathanson and previous work by him, Chen--Yang, Kiss--S\`andor--Yang etc. focussed on $G =\mathbb{Z}$. In the higher rank case, recent work by the authors treated a class of subsets, namely the eventually periodic sets. However, for infinite subsets, not of the above type, the question of existence or inexistence of minimal complements is open. In this article, we study subsets which are not eventually periodic. We introduce the notion of ``spiked subsets" and give necessary and sufficient conditions for the existence of minimal complements for them. This provides a partial answer to a problem of Nathanson.
\end{abstract}

\maketitle
\vspace*{-5mm}
\tableofcontents

\vspace*{-8mm}
\section{Introduction}

\subsection{Motivation and Background}

Let $(G,+)$ be an abelian group and $W\subseteq G$ be a nonempty subset. A nonempty set $W'\subseteq G$ is said to be an \textit{additive complement} to $W$ if $$W + W' = G.$$
Given any non-empty set, it is clear that a complement always exist (namely, take the whole group to be the complement). So the set of all additive complements to a particular set $W$, is non-empty. Also, the above set of complements is partially ordered (by inclusion). The whole group $G$ is a maximal (and also maximum) element. However, the question of existence of a minimal element is tricky. It depends on the structure of the set $W$.
Nathanson in $2011$ introduced the notion of a minimal complement.

\begin{definition}[Minimal complement]
A complement $W'$ to $W$ is said to be \textnormal{minimal} if no proper subset of $W'$ is a complement to $W$, i.e., 
$$W + W' = G \,\text{ and }\, W + (W'\setminus \lbrace w'\rbrace)\subsetneq G \,\,\, \forall w'\in W'.$$
\end{definition}

It was shown by Nathanson (see \cite[Theorem 8]{NathansonAddNT4}) that for a non-empty, finite subset $W$ in the additive group $\mathbb{Z}$, any complement to $W$ has a minimal complement. In the same paper he posed some general questions on existence of minimal complements. 

\begin{questionIntro}
\cite[Problem 11]{NathansonAddNT4}
\label{nathansonprob11}
``Let $W$ be an infinite set of integers. Does there exist a minimal
complement to $W$? Does there exist a complement to $W$ that does not contain a
minimal complement?''
\end{questionIntro}

\begin{questionIntro}
\cite[Problem 12]{NathansonAddNT4}
\label{nathansonprob12}
``Let $G$ be an infinite group, and let $W$ be a finite subset of $G$. Does
there exist a minimal complement to $W$? Does there exist a complement to $W$ that
does not contain a minimal complement?''
\end{questionIntro}

\begin{questionIntro}\footnote{The statement of Question \ref{nathansonprob13}  as formulated here is different from the exact statement of \cite[Problem 13]{NathansonAddNT4}. The two formulations are equivalent.}
\cite[Problem 13]{NathansonAddNT4}\label{nathansonprob13}
Let $G$ be an infinite group, and let $W$ be an infinite subset of $G$.
\begin{enumerate}[(a)]
\item Does there exist a minimal complement to $W$? 
\item Does every complement of $W$ contain a minimal complement? 

\end{enumerate}
\end{questionIntro}

Chen and Yang were the first to give a partial answer to Question \ref{nathansonprob11}. They constructed two infinite sets $W_{1}$ and $W_{2}$ such that, every complement to $W_{1}$ always contains a minimal complement whereas, there exists a complement to $W_{2}$ that does not contain a minimal complement \cite{ChenYang12}. Later, Kiss--S\`andor--Yang constructed a class of infinite sets in $\mathbb{Z}$ which they called ``eventually periodic sets" and gave necessary and sufficient conditions for these sets to have a minimal additive complement in $\mathbb{Z}$ \cite{KissSandorYangJCT19}.

All the above dealt with Question \ref{nathansonprob11}. In a recent work, we answered Question \ref{nathansonprob12} completely and gave the first partial solution to Question \ref{nathansonprob13} (see \cite[Theorems 2.1, 4.9, 4.12, 4.23]{MinComp1}). For the latter, a part of the method was to define a notion of eventually periodic sets in the free abelian group $\mathbb{Z}^{r}$ of rank $r\geq 1$. Further, taking the product of these eventually periodic sets in $\mathbb{Z}^{r}$ with finite groups, we gave conditions for sets in arbitrary finitely generated abelian groups to have additive minimal complements (see \cite[Theorem 5.6]{MinComp1}). However, not all infinite subsets in an arbitrary finitely generated abelian group (or even a free abelian group of finite rank) are of the above form. Thus, for a complete understanding of the question we have to consider subsets which don't fall in the purview of the above.

\subsection{Statement of results}
First, we show that minimal complements of sets do not ``escape" when the sets are confined inside some proper subgroup of a larger group.
\begin{theoremIntro}[Theorem \ref{Thm:MinCompInSubgp}]
Let $G$ be a group and $W$ be a nonempty subset of $G$. Then the following statements are equivalent.  
\begin{enumerate}
	\item $W$ admits a minimal complement in $G$.
	\item $W$ admits a minimal complement in $H$ for any subgroup $H$ of $G$ containing $W$. 
	\item $W$ admits a minimal complement in $H$ for some subgroup $H$ of $G$ containing $W$. 
\end{enumerate}
\end{theoremIntro}

This is a crucial and satisfactory phenomenon in arbitrary groups as it allows one to concentrate only inside the subgroup generated by the set (the existence of whose minimal complement we seek to investigate). As an example  suppose $G= \mathbb{Z}^{d}, d>1$ and $W\subset \mathbb{Z}^{r} \subset G$ with $r<d$. Then a minimal complement of $W$ exists in $G$ if and only if $W$ admits a minimal complement in $\mathbb{Z}^{r}$.

Next, we study subsets truncated along vertical lines. We show that they do not admit a minimal complement. This is Prop. \ref{Prop:Truncated2D}. This allows us to derive the fact that upper or lower portions of graphs of functions in free abelian groups of dimension $d$ do not admit minimal complements. This phenomenon also holds under translations and right rotations [see Prop. \ref{Prop:Truncated30DSlant} and Cor. \ref{Cor:SideParabola}, Cor. \ref{Cor:sideWhatever}]. In $\mathbb{Z}^{2}$ we show that it also holds under rational rotations [see Cor. \ref{Cor:LowerParabolaRota} and Cor. \ref{Cor:LowerWhateverRota}]. As a particular case, truncation along horizontal lines give the same result.

What happens when the subsets contain lines? This brings us to the notion of a moderation function [see Def. \ref{moderation}] and spiked subsets [see Def. \ref{Spiked subsets}]. We show that spiked subsets belong to the class of sets which are not eventually periodic [see Prop. \ref{Spiked vs eventually periodic}]. We first study spiked subsets in the free abelian group $\mathbb{Z}^{d} (d\geqslant 2)$ and show a sufficient condition for the existence of minimal complements and its relation with the moderation function. Specifically,

\begin{theoremIntro}
	[Theorem \ref{Thm:LinesNoise}]
	Let $u:\bbZ^k\to \bbZ$ be a function. Let $\calB$ be a subset of $\bbZ^k$ admitting a minimal complement $M$ in $\bbZ^k$. Then $u$ admits a moderation and for any moderation $v$ of $u$, the subset 
	$$M_v:=\{(x, v(x))\,|\, x\in M\}$$
	of $\bbZ^{k+1}$ is a minimal complement of any subset $X$ of $\bbZ^{k+1}$ satisfying 
	\begin{equation}
	\label{Eqn:ContainmentX}
	\calB\times \bbZ
	\subseteq
	X
	\subseteq 
	(\calB\times \bbZ )
	\bigsqcup 
	\left(
	\sqcup 
	_{x\in \bbZ^k\setminus \calB}
	\left(
	\{
	x
	\}
	\times 
	(-\infty, u(x))
	\right)
	\right).
	\end{equation}
\end{theoremIntro}

See also Prop. \ref{Prop:LinesNoise}. From the above theorem and proposition we give classes of (via corollaries and remarks - see Cor. \ref{Cor4.8}, Cor. \ref{Cor4.9}, Remark \ref{Rk:PolyModeration}) subsets of the free abelian groups admitting minimal complements.

Now, we shift our attention to general abelian groups. We define the notion of spiked subsets in this setting and show a necessary and sufficient condition for them to have minimal complements.

\begin{theoremIntro}[Theorem \ref{Thm:NecessarySuffi}]
	Let $G_1, G_2$ be subgroups of $\calG$ such that the intersection $G_1\cap G_2$ is trivial and $G_2$ is free of positive rank. Let $\calB$ be a nonempty subset of $G_1$. Then the following statements are equivalent. 
	\begin{enumerate}
		\item $\calB$ admits a minimal complement in $\calG$.
		\item $\calB$ admits a minimal complement in $G_1$.
		\item Any $(u, \varphi)$-bounded spiked subset of $\calG$ with respect to $G_1, G_2$, having $\calB$ as its base, admits a minimal complement in $G_1\times G_2$ which is the graph of any $\varphi$-moderation of $u$ restricted to some subset of $G_1$. 
		\item Some $(u, \varphi)$-bounded spiked subset of $\calG$ with respect to $G_1, G_2$, having $\calB$ as its base, admits a minimal complement in $G_1\times G_2$ which is the graph of any $\varphi$-moderation of $u$ restricted to some subset of $G_1$. 
	\end{enumerate}
\end{theoremIntro}
Cor. \ref{Cor:LowerWhateverYAxisRota} allows one to rotate (rational rotation) the sets in $\mathbb{Z}^{2}$ without hampering the existence of the minimal complement. Then, we show Cor. \ref{Prop:LinesNoisek1k2} which allows us to construct (and conclude) other sets admitting minimal complements [see Example \ref{Cor5.9}]. Finally, we show a more general result:
\begin{theoremIntro}[Theorem \ref{Thm5.11}]
	Let $G_1, G_2$ be subgroups of an abelian group $\calG$ such that $G_2$ is free of positive rank, $G_1\cap G_2$ is trivial. Let $\calB$ be a subset of $G_1$ with minimal complement $M$ in $G_1$. Let $u: G_1\to G_2$ be a function and 
	$$\varphi: G_2 \xra{\sim} \bbZ^{\rk G_2}$$ 
	be an isomorphism. Let $G_2'$ be a subgroup of $G_2$ of finite index and $g_2$ be an element of $G_2$. Let $X$ be a subset of $\calG$ containing $\calB G_2$ and contained in 
	$$
	\calB G_2
	\bigsqcup 
	\left(
	\bigsqcup_{g_1\in G_1\setminus \calB}
	g_1\cdot
	\left(
	\varphi^\mo \left(\bbZ^{\rk G_2}_{<\varphi(u(g_1))}
	\right)
	\right)
	\right)
	\bigsqcup 
	\left(
	\bigsqcup_{g_1\in G_1\setminus \calB}
	g_1\cdot
	\left(
	\left(
	\varphi^\mo \left(\bbZ^{\rk G_2}_{\not<\varphi(u(g_1))}
	\right)
	\right)
	\cap 
	\left(
	G_2\setminus g_2 G_2'
	\right)
	\right)
	\right).
	$$
	Then $u$ admits a $G_2'$-valued $\varphi$-moderation and the graph $M_v$ of the restriction $v'|_M$ of such a moderation $v'$ to $M$, i.e.,
	$$
	M_{v'}
	=
	\{(x,v'(x))\,|\, x\in M\}
	$$
	is a minimal complement of $X$ in $G_1\times G_2$. 
\end{theoremIntro}

This finishes our analysis of truncated subsets and non-truncated (spiked) subsets in finitely generated abelian groups providing a comprehensive description of sets which were unknown to have minimal complements or not. The claimed partial answer to \cite[Problem 13]{NathansonAddNT4} follows.

\subsection{Plan of the paper}
The article is divided into $6$ main sections.
\begin{enumerate}
	\item Introduction --- Here we give the background, motivation and results.
	\item Generalities --- This section deals with detailed analysis of Questions \ref{nathansonprob13}(a), \ref{nathansonprob13}(b) and also of minimal complements of sets which are confined inside some proper subgroup [Theorem \ref{Thm:MinCompInSubgp}].
	\item Subsets truncated along lines --- Here we study the inexistence of minimal complements for subsets truncated along lines.
	\item Subsets containing lines --- This section introduces the concept of moderation functions and spiked subsets in free abelian groups. Theorem \ref{Thm:LinesNoise} gives a sufficient condition for existence of minimal complements.
	\item Spiked subsets in general abelian groups --- This section continues with the study of spiked subsets and extends the concept to general abelian groups. It also contains the necessary and sufficient condition for existence of minimal complements of spiked subsets [Theorem \ref{Thm:NecessarySuffi}] and the more general Theorem \ref{Thm5.11}.
	\item Concluding remarks --- We conclude with several new constructions of sets in higher rank abelian groups having minimal complements.
\end{enumerate}

\vspace*{5mm}
\section{Generalities}
\label{sec2}
We begin Section \ref{sec2} by stating a well known and useful lemma about the behaviour of minimal complements when the set is translated by an element of the group. The proof is elementary. For completeness sake, we furnish the proof here.

\begin{lemma}[\cite{NathansonAddNT4}]
	Minimal complements are preserved under translations.
\end{lemma}

\begin{proof}
If $M$ is a minimal complement of a subset $W$ of a group $G$, then for any $g\in G$, $Mg$ is a complement to $W$ since $WMg = (WM) g = Gg= G$. Moreover, $Mg$ is also a minimal complement to $W$. Otherwise, for some $m\in M$, $Mg\setminus \{mg\} $ would be a complement to $W$, which would imply $G = W (Mg\setminus \{mg\}) = W( (M\setminus \{m\} )g)$, implying $G = W (M\setminus\{m\}$, which is absurd. 
\end{proof}

Now, we discuss the possible answers to Questions \ref{nathansonprob13}(a), (b) (see Table \ref{Table1}). We give an example in Lemma \ref{Lemma: YY} indicating that Question \ref{nathansonprob13}(a),(b) could have answers in affirmative and negative respectively (for the same subset $W$ of $G = \bbZ^2$).

\begin{center}
\begin{longtable}{|p{0.1\textwidth} p{0.1\textwidth} | p{0.7\textwidth} |}
\hline
\multicolumn{2}{|c|}{Answers to } & Examples and Remarks \\
Question \ref{nathansonprob13}(a) & Question \ref{nathansonprob13}(b) & \\
\hline
Yes & No & Consider the set $\bbZ^2 \setminus \{ (0,n)\,|\, n\in \bbN\}$. Each of the sets $\{(0,0), (1,0)\}$, $\{ (0,n)\,|\, n\in \bbN\}$ is a complement to it. The first is a minimal complement. However, the second i.e., $\{ (0,n)\,|\, n\in \bbN\}$ does not contain any minimal complement (see Lemma \ref{Lemma: YY}).\\ 
\hline
Yes & \hspace*{-1.1cm} $\Longleftarrow\,\,$  Yes &
Consider the case when $W$ is an infinite subgroup of $G$ (see \cite[Proposition 4.1]{MinComp1}). 
Moreover, if Question \ref{nathansonprob13}(b) has an answer in affirmative, then the Question \ref{nathansonprob13}(a) necessarily has an answer in affirmative. \\ 
\hline
No & \hspace*{-1.1cm} $\Longrightarrow\,\,$  No & Consider the subset of $\bbZ^2$ consisting of points lying below the parabola $Y=X^2$ (see Corollary \ref{Cor:LowerParabola}, \ref{Cor:LowerWhatever}). Moreover, if Question \ref{nathansonprob13}(a) has an answer in negative, then the Question \ref{nathansonprob13}(b) necessarily has an answer in negative. \\ 
\hline
No & Yes & It is clear that such a situation does not arise.\\
\hline
\caption{Question \ref{nathansonprob13}}
\label{Table1}
\end{longtable}
\end{center}

\vspace*{-10mm}
\begin{lemma}
\label{Lemma: YY}
Let $W$ denote the subset 
$$W := \bbZ^2\setminus \{ (0,n)\,|\, n\in \bbN\}\subset \bbZ^2.$$ Any subset $S$ containing two points of $\bbZ^2$ with distinct $x$-coordinates is a complement to $W$ and a minimal complement to $W$. Any infinite subset 
$$C\subseteq \{(0,n)\,|\, n\in \bbN\}$$ is a complement to $W$. Consequently, $C$ is not a minimal complement to $W$. 
\end{lemma}

\begin{proof}
Let $(a,b), (c, d)$ be two points of $\bbZ^2$ with $a\neq c$. Note that $W$ contains the line $\calL$ defined by $X=c-a$. Then $(a,b) + \calL$ is equal to the line $X=c$, and is contained in $(a, b) + W$. Since $W$ contains the complement of the line $X=0$ in $\bbZ^2$, the set $(c,d)+W$ contains the complement of the line $X=c$ in $\bbZ^2$. Hence $\{(a,b), (c,d)\}$ is a complement to $W$. Since $W + (a,b)$ (resp. $W+(c,d)$) does not contain the point $(a,b+1)$ (resp. $(c, d+1)$), it follows that $\{(a,b), (c,d)\}$ is a minimal complement to $W$. 

Since any infinite subset $C$ of $\{(0,n)\,|\, n\in \bbN\}$ is a complement to $\{(x,y)\in \bbZ^2\,|\, y < 0\}$, the set $C$ also a minimal complement to $W$. So any infinite subset of $C$ is also a complement to $W$, and hence $C$ is not a minimal complement to $W$.
\end{proof}

\subsection{Minimal complements of confined subsets} 
A natural question is to ask what happens to the minimal complement of a set when the set itself is confined inside some proper subgroup of a larger group. It will be a satisfactory phenomenon if the minimal complement does not ``escape" in the sense 
that if we consider a subset of $G$ which is contained in a proper subgroup, then the existence of its minimal complement in $G$ is equivalent to the existence of its minimal complement in some (or any) proper subgroup which contains the set. We show that this is indeed the case.

\begin{theorem}
\label{Thm:MinCompInSubgp}
Let $G$ be a group and $W$ be a nonempty subset of $G$. Then the following statements are equivalent.  
\begin{enumerate}
\item $W$ admits a minimal complement in $G$.
\item $W$ admits a minimal complement in $H$ for any subgroup $H$ of $G$ containing $W$. 
\item $W$ admits a minimal complement in $H$ for some subgroup $H$ of $G$ containing $W$. 
\end{enumerate}
\end{theorem}

\begin{proof}
Let $M$ be a minimal complement of $W$ in $G$. Let $H$ be a subgroup of $G$ containing $W$. Let $\{M_\lambda\}_{\lambda\in \Lambda}$ denote the equivalence classes of $M$ with respect to the equivalence relation defined by 
$$m_1\sim m_2\Longleftrightarrow m_1m_2^\mo \in H \text{ for elements } m_1, m_2\in M.$$ Note that $G$ is equal to the disjoint union $\sqcup_{\lambda\in \Lambda} W M_\lambda$. Let $\lambda_0$ be the unique element of $\Lambda$ such that $M_{\lambda_0}\subseteq H$, i.e., $$M_{\lambda_0} = M\cap H.$$ Note that it is a minimal complement to $W$ in $H$. Hence the second statement holds. 

It remains to prove that the third statement implies the first statement. Assume that the third statement holds, i.e., for some subgroup $H$ of $G$ containing $W$, the set $W$ admits a minimal complement $M$ in $H$. Let $\{g_\lambda\}_{\lambda\in \Lambda}$ be a set of right coset representatives of $H$ in $G$. Note that $\cup_{\lambda\in \Lambda} M g_\lambda$ is a complement of $W$ in $G$. Since $G$ is the disjoint union of the subsets $WMg_\lambda$ and $M$ is a minimal complement to $W$ in $H$, it follows that $\cup_{\lambda\in \Lambda} M g_\lambda$ is a minimal complement of $W$ in $G$.
\end{proof}

\section{Subsets truncated along lines}
In this section, we study subsets of $\mathbb{Z}^{2}$ which are bounded by graphs of functions. 
\begin{proposition}
\label{Prop:Truncated2D}
Let $W$ be a subset of $\bbZ^2$ of the form 
$$W:=\cup_{a\in A} \{a\}\times (-\infty, u(a))$$ where $A$ is a nonempty subset of $\bbZ$ and $u:A\to \bbZ$ denotes a function. Then $W$ does not admit any minimal complement in $\bbZ^2$.
Similarly, no subset of $\bbZ^2$ of the form $$\cup_{a\in A} \{a\}\times (u(a), \infty)$$ admits a minimal complement in $\bbZ^2$ where $A$ is a nonempty subset of $\bbZ$ and $u:A\to \bbZ$ denotes a function. 
\end{proposition}

\begin{proof}
Let $C$ be a complement of $W$ in $\bbZ^2$. Fix an element $c$ of $C$. Since $W$ is strictly contained in $\bbZ^2$, any of its translates by an element of $\bbZ^2$ is also strictly contained in $\bbZ^2$. So  $C$ contains at least two elements. Hence $C\setminus\{c\}$ is nonempty. Note that for each $w$ in $W$, there exists a positive integer $n_w$ such that 
$$w+(0,n_w)\not\subseteq W.$$ Then, $$c + w + (0,n_w)\not\subseteq c+W$$ for any $w\in W$. So for each $w\in W$, there exists an element $c_w\in C$ such that $$c+ w + (0,n_w)\subseteq c_w + W.$$ Since for any positive integer $n$, the set $W$ contains $w - (0,n)$ for each $w\in W$, it follows that $$c+W\subseteq c_w + W.$$ Consequently, $C\setminus\{c\}$ is also a complement to $W$. Hence $W$ admits no minimal complement. 

Note that a subset of $\bbZ^2$ admits a minimal complement in $\bbZ^2$ if and only if its image under any automorphism of $\bbZ^2$ admits a minimal complement. Moreover, $W$ is the union of certain lines truncated above if and only if $\sigma(W)$ is the union of certain lines truncated below where $\sigma$ denotes the automorphism of $\bbZ^2$ given by 
$$\sigma: (x,y)\mapsto (x, -y).$$ Hence the second statement follows from the first statement.
\end{proof}

From Proposition \ref{Prop:Truncated2D}, we obtain the following corollary. 

\begin{corollary}
\label{Cor:LowerParabola}
Let $W_+$ (resp. $W_-$) denote the subset of $\bbZ^2$ consisting of points lying above (resp. below) the parabola $Y=X^2$, i.e.,
\begin{align*}
	W_+ & = 
	\{
	(x, y)\in \bbZ^2 \,| \, y > x^2
	\}\\
	W_- & = 
	\{
	(x, y)\in \bbZ^2 \,| \, y < x^2
	\}
\end{align*}
None of the subsets $W_+, W_-$ admit any minimal complement in $\bbZ^2$  \textnormal{[see Fig.\ref{Fig1}]}. 
\end{corollary}

\begin{figure}[h]
	\centering
	\begin{subfigure}{.5\textwidth}
		\centering
		\includegraphics[scale=0.08]{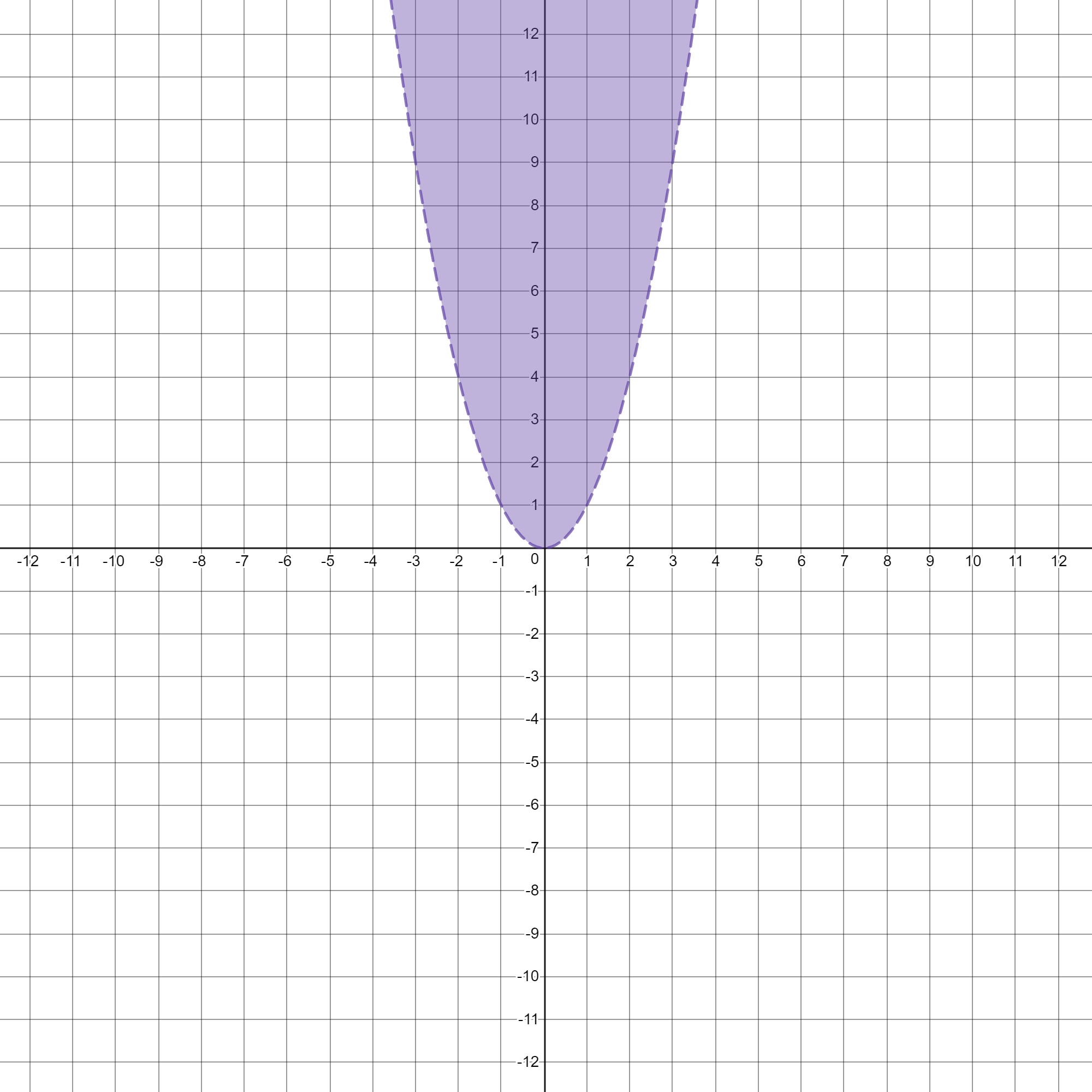}
		\caption{$W_{+}$}
		\label{}
	\end{subfigure}%
	\begin{subfigure}{.5\textwidth}
		\centering
		\includegraphics[scale=0.08]{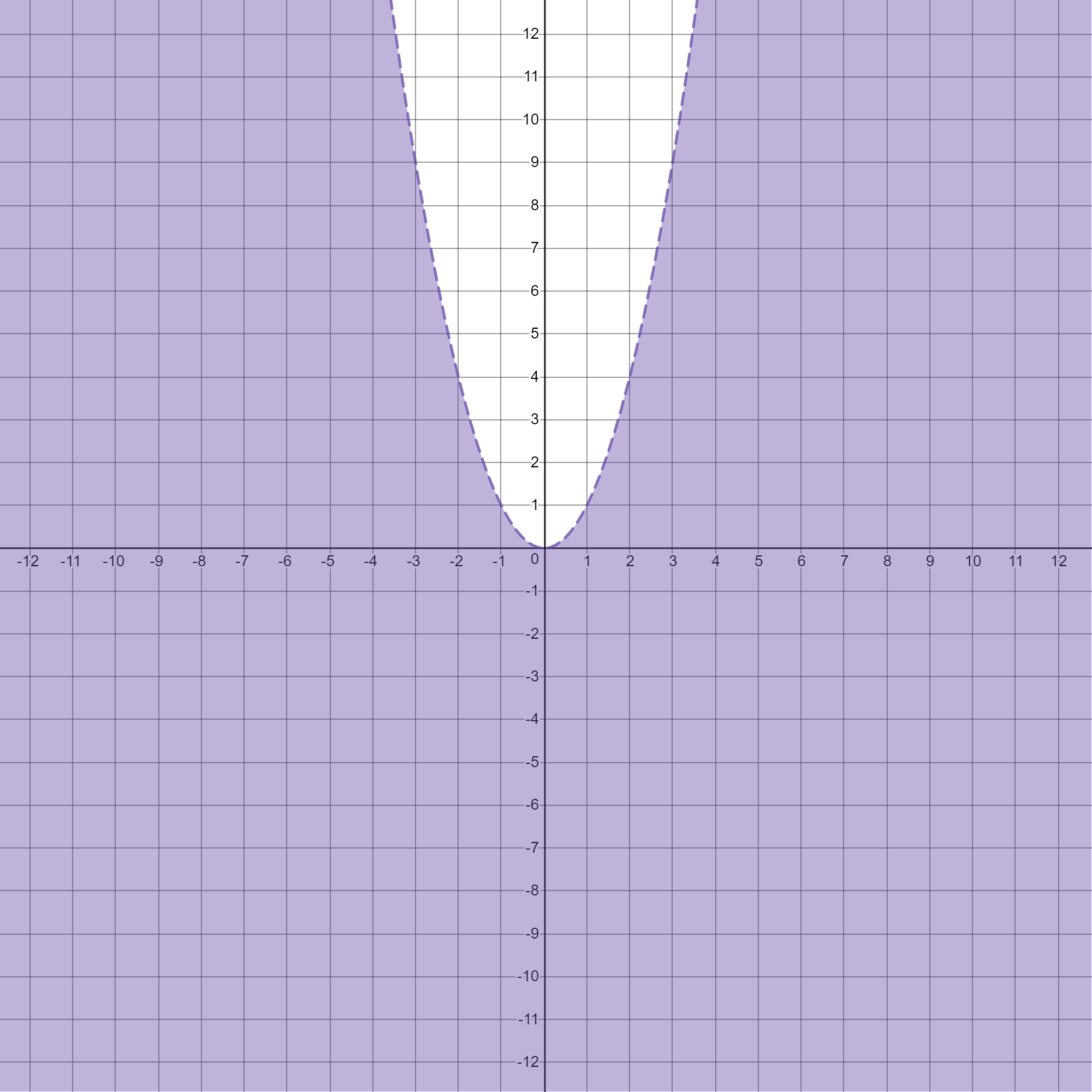}
		\caption{$W_{-}$}
		\label{}
	\end{subfigure}
	\caption{Cor. \ref{Cor:LowerParabola}. $W_{+}$ and $W_{-}$ don't admit any minimal complement in $\mathbb{Z}^{2}$}
	\label{Fig1}
\end{figure}

More generally, we have the following result. 

\begin{corollary}
\label{Cor:LowerWhatever}
Let $f:\bbZ \to \bbR$ be a function and let $W_\pm$ denote the subset of $\bbZ^2$ defined by
\begin{align*}
W_+ & = 
\{
(x, y)\in \bbZ^2 \,| \, y > f(x)
\}\\
W_- & = 
\{
(x, y)\in \bbZ^2 \,| \, y < f(x)
\}
\end{align*}
None of the subsets $W_+, W_-$ admit any minimal complement in $\bbZ^2$ \textnormal{[see Fig.\ref{Fig2} and Fig. \ref{Fig3}]}.
\end{corollary}

\begin{figure}[h]
	\centering
	\begin{subfigure}{.5\textwidth}
		\centering
		\includegraphics[scale=0.08]{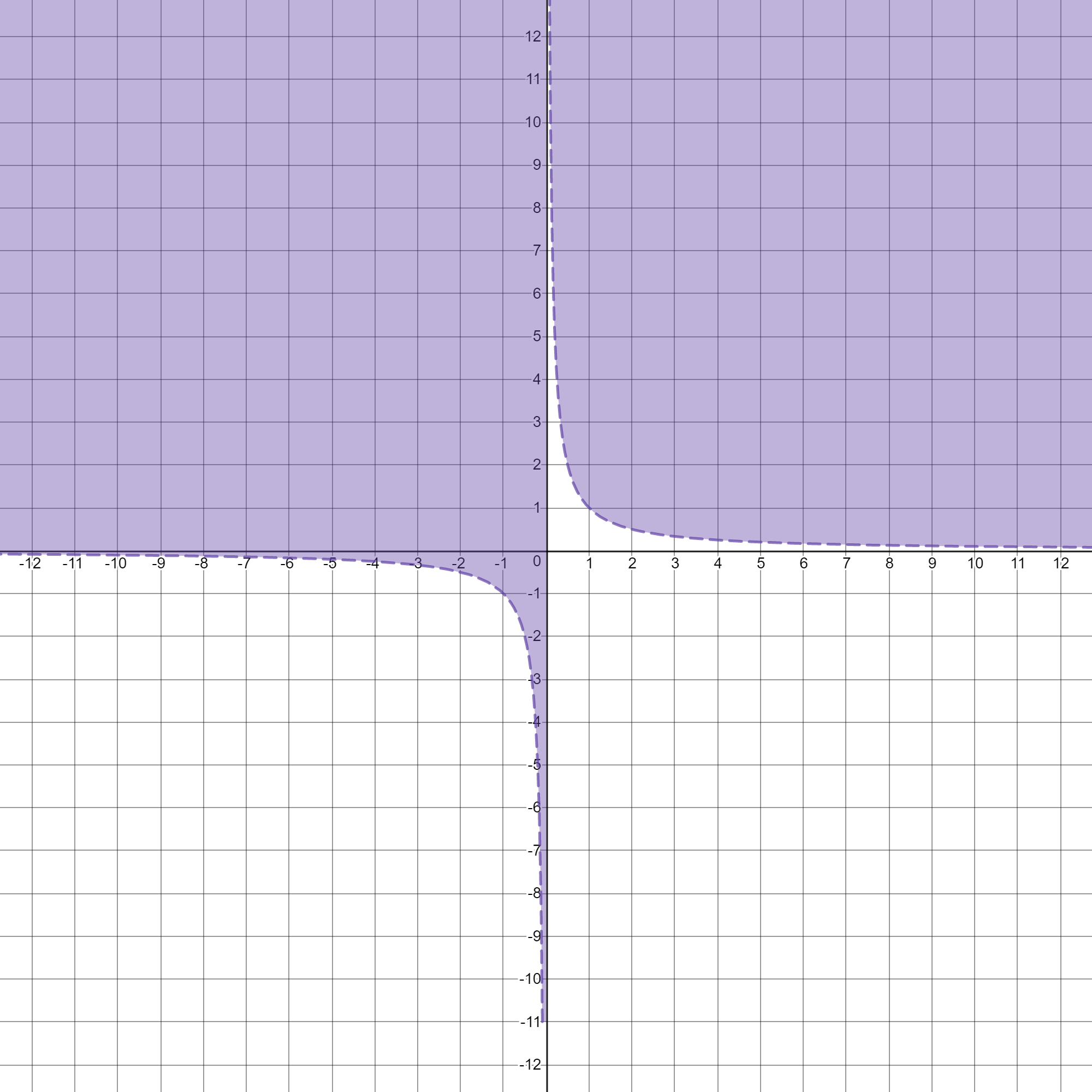}
		\caption{$W_{+}: f(x) = x^{-1}, x\neq 0, f(0)=t$}
		\label{}
	\end{subfigure}%
	\begin{subfigure}{.5\textwidth}
		\centering
		\includegraphics[scale=0.08]{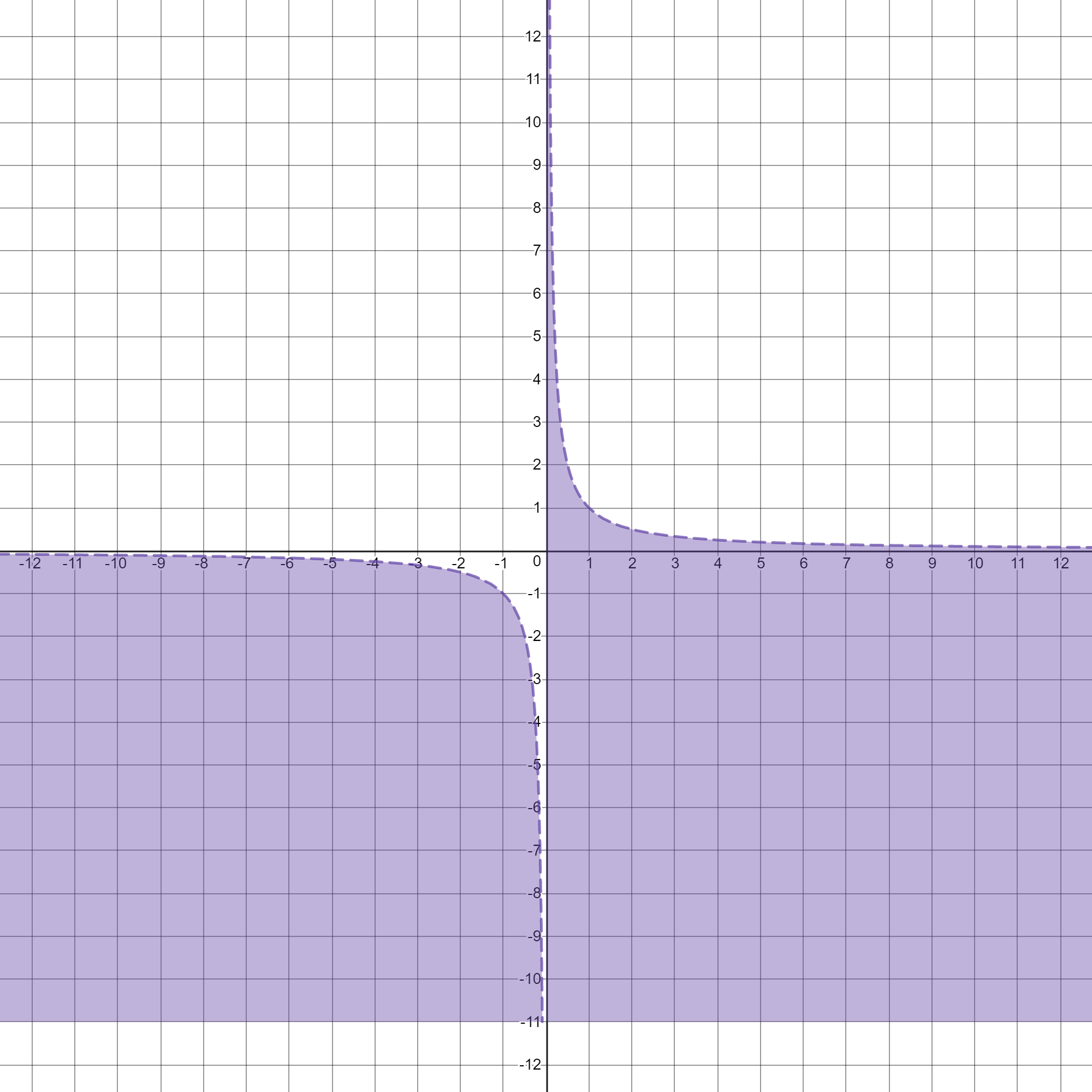}
		\caption{$W_{-}: f(x) = x^{-1}, x\neq 0, f(0)=t$}
		\label{}
	\end{subfigure}
	\caption{Cor. \ref{Cor:LowerWhatever}. $W_{+}$ and $W_{-}$ with a parameter $t$. For any $t\in \mathbb{Z}$, they don't admit any minimal complement in $\mathbb{Z}^{2}$. In the figure, graph of $t=-11$ is shown.}
	\label{Fig2}
\end{figure}

\begin{figure}[h]
	\centering
	\begin{subfigure}{.5\textwidth}
		\centering
		\includegraphics[scale=0.08]{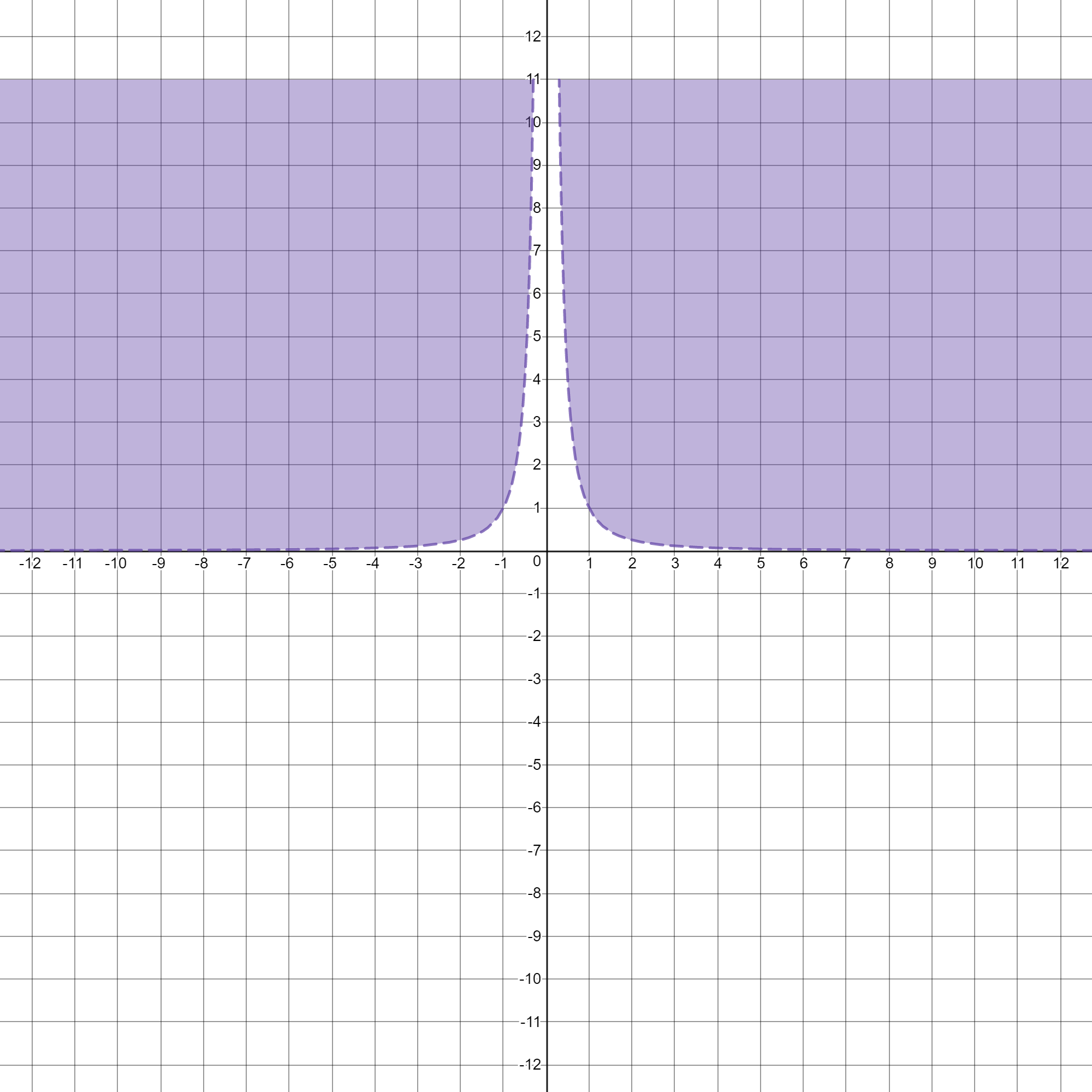}
		\caption{$W_{+}: f(x) = x^{-2}, x\neq 0, f(0)=t$}
		\label{}
	\end{subfigure}%
	\begin{subfigure}{.5\textwidth}
		\centering
		\includegraphics[scale=0.08]{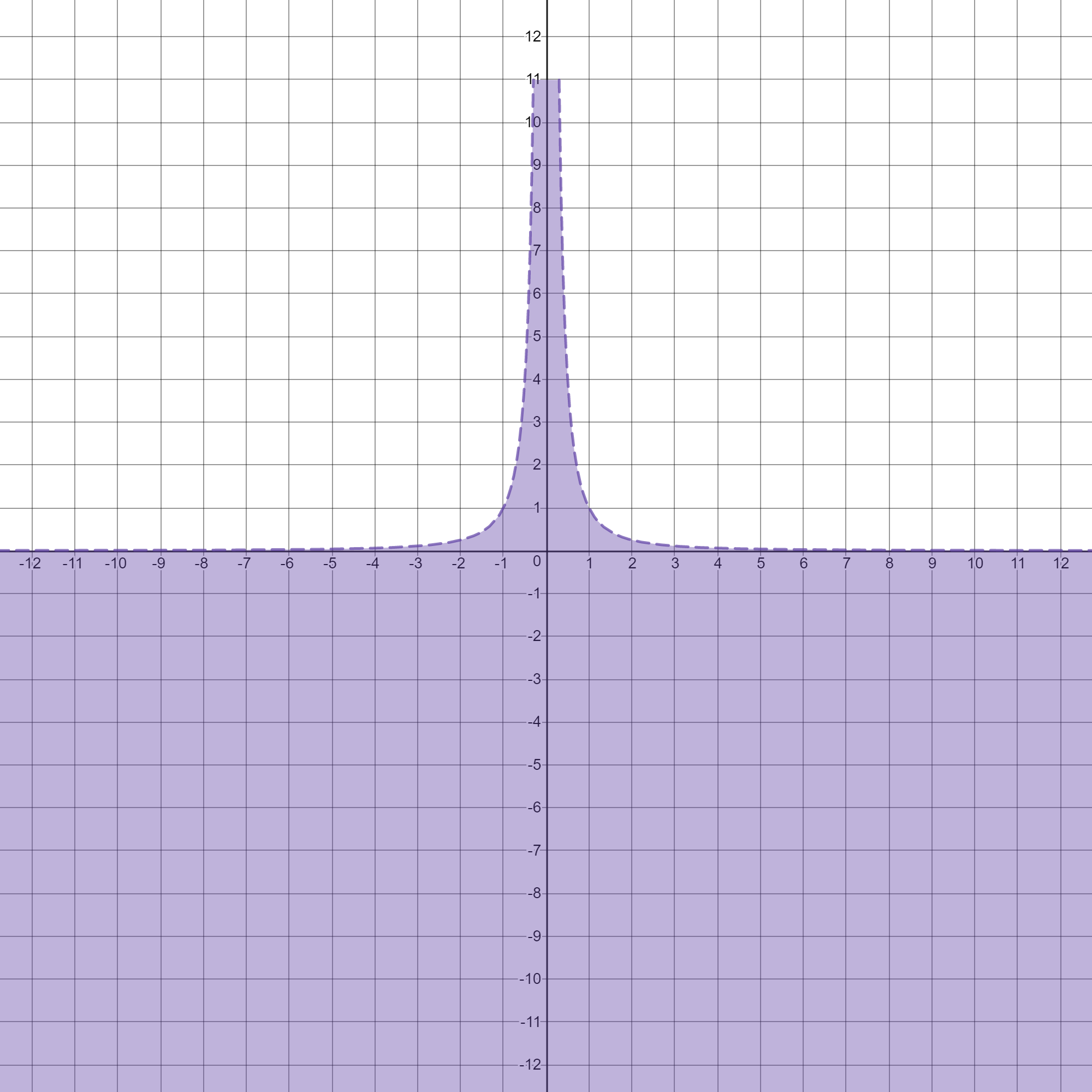}
		\caption{$W_{-}: f(x) = x^{-2}, x\neq 0, f(0)=t$}
		\label{}
	\end{subfigure}
	\caption{Cor. \ref{Cor:LowerWhatever}. $W_{+}$ and $W_{-}$ with a parameter $t$. For any $t\in \mathbb{Z}$, they don't admit any minimal complement in $\mathbb{Z}^{2}$. In the figure, graph of $t=11$ is shown.}
	\label{Fig3}
\end{figure}

\begin{proof}
Since $W_+$ is the union of the truncated lines 
$$\{n\} \times (\lfloor f(n)\rfloor , \infty), n\in \mathbb{Z}$$  and $W_-$ is the union of the truncated lines $$\{n\} \times (-\infty, \lfloor f(n)\rfloor + \lceil f(n)- \lfloor f(n)\rfloor \rceil), n\in \mathbb{Z},$$ by Proposition \ref{Prop:Truncated2D}, it follows that none of $W_+, W_-$ admit any minimal complement in $\bbZ^2$.
\end{proof}

\begin{proposition}
\label{Prop:Truncated30DSlant}
Let $G_1, G_2$ be abelian groups and $G_2$ be free of positive rank. Let 
$$\varphi:G_2 \xra{\sim} \bbZ^{\rk G_2}$$ be an isomorphism, $A$ be a subset of $G_1$ and $u:G_1\to G_2$ be a function. Then none of the subsets
\begin{align*}
	& \cup_{a\in A} \{a\}\times \varphi^\mo \left(\bbZ^{\rk G_2}_{<\varphi(u(a))}\right),\\
	& \cup_{a\in A} \{a\}\times \varphi^\mo \left(\bbZ^{\rk G_2}_{>\varphi(u(a))}\right) 
\end{align*}
of $G_1\times G_2$ admit any minimal complement in $G_1\times G_2$. 
\end{proposition}

\begin{proof}
The proof is similar to the proof of Proposition \ref{Prop:Truncated2D}.
\end{proof}

From Proposition \ref{Prop:Truncated30DSlant}, we obtain the following corollaries:

\begin{corollary}
	\label{Cor:SideParabola}
	Let $W_+$ (resp. $W_-$) denote the subset of $\bbZ^2$ consisting of points lying on the right (resp. left) of the parabola $X=Y^2$, i.e.,
	\begin{align*}
	W_\pm & = 
	\{
	(x, y)\in \bbZ^2 \,| \, \pm(x-y^{2})>0
	\}.
	\end{align*}
	None of the subsets $W_+, W_-$ admit any minimal complement in $\bbZ^2$ \textnormal{[see Fig.\ref{Fig4}]}. 
\end{corollary}

\begin{figure}[h]
	\centering
	\begin{subfigure}{.5\textwidth}
		\centering
		\includegraphics[scale=0.08]{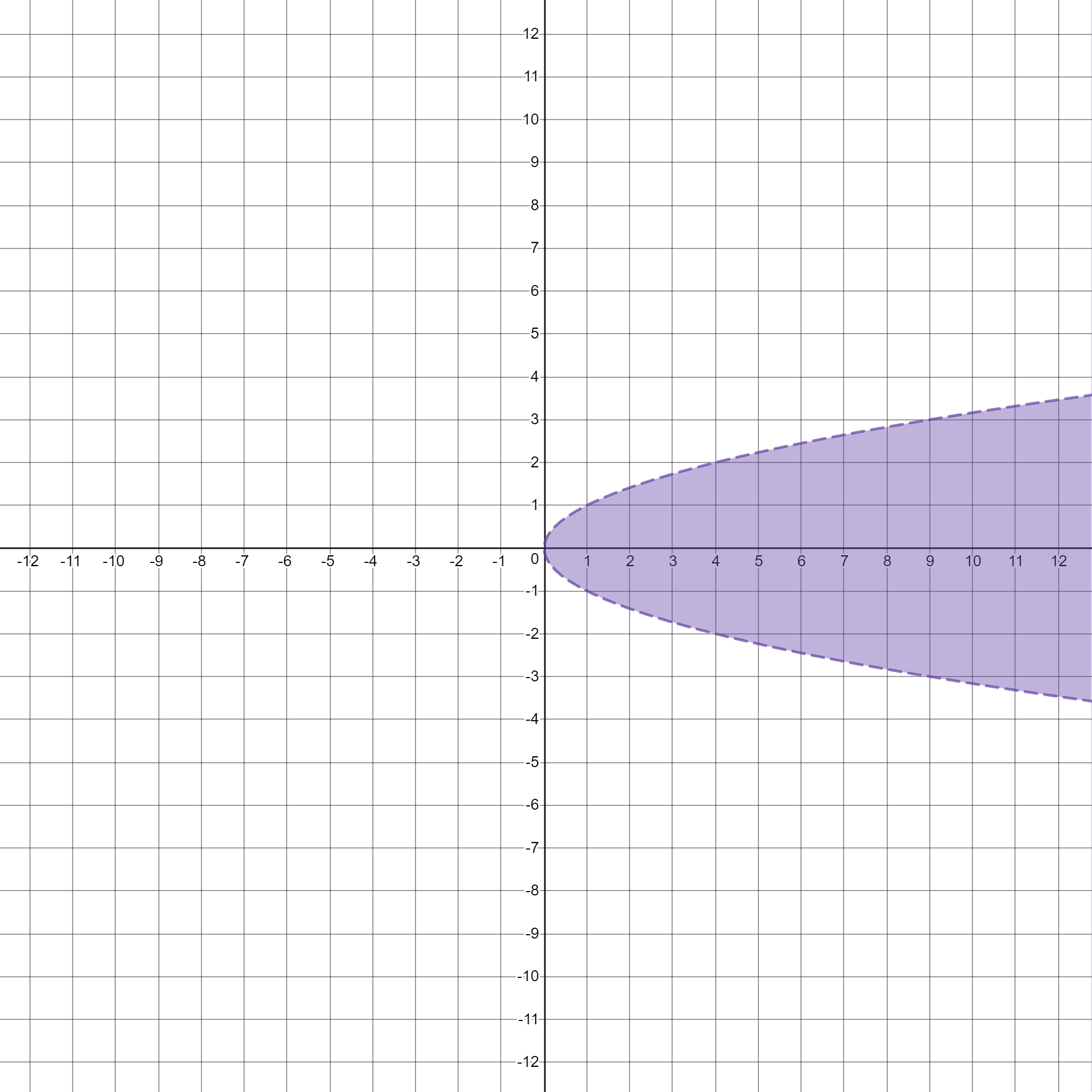}
		\caption{$W_{+}$}
		\label{}
	\end{subfigure}%
	\begin{subfigure}{.5\textwidth}
		\centering
		\includegraphics[scale=0.08]{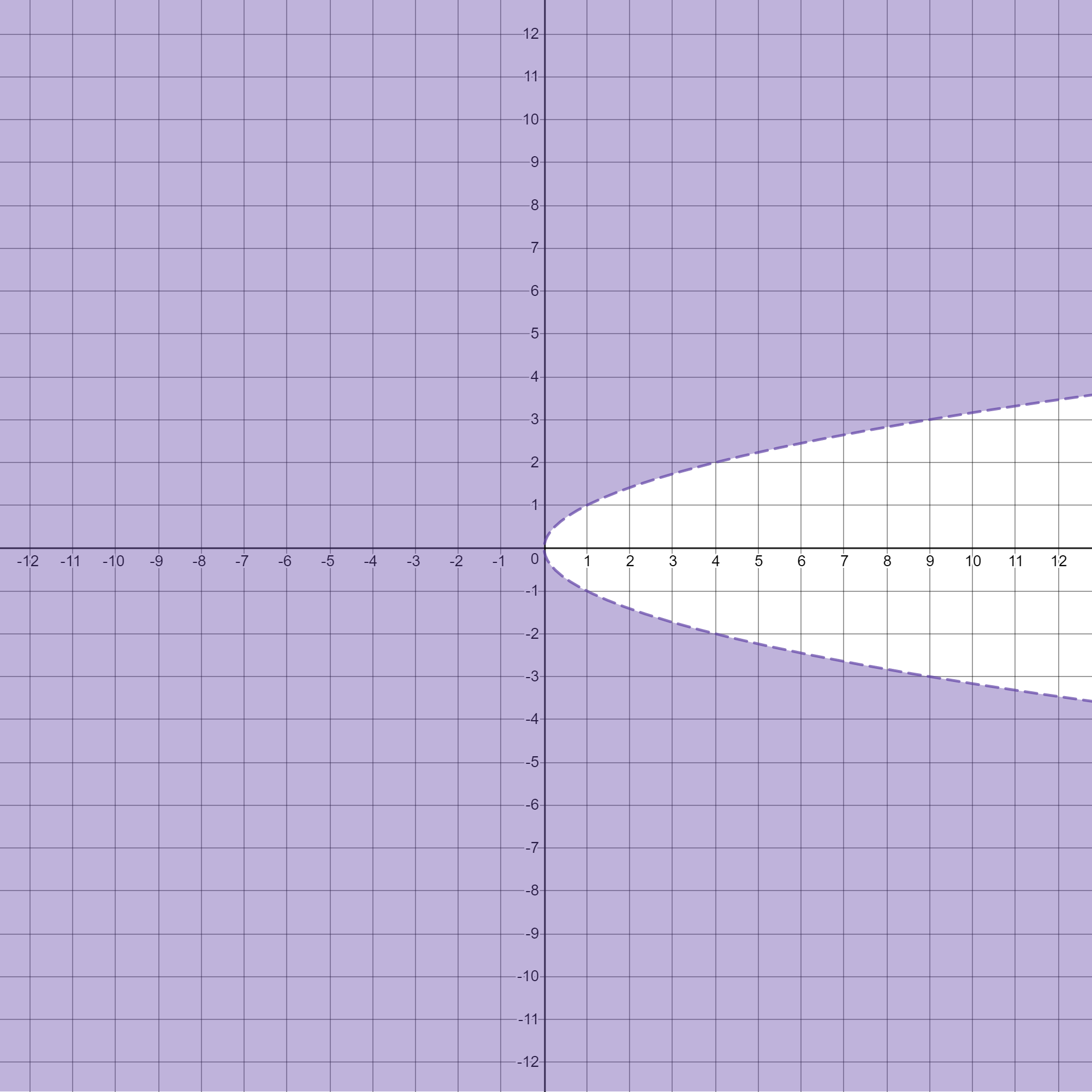}
		\caption{$W_{-}$}
		\label{}
	\end{subfigure}
	\caption{Cor. \ref{Cor:SideParabola}. $W_{+}$ and $W_{-}$ don't admit any minimal complement in $\mathbb{Z}^{2}$}
	\label{Fig4}
\end{figure}

\begin{corollary}
\label{Cor:LowerParabolaRota}
Let $W_+$ (resp. $W_-$) denote the subset of $\bbZ^2$ consisting of points lying `above' (resp. `below') the parabola $Y=X^2$ rotated clockwise by $45^\circ$, i.e.,
\begin{align*}
	W_+ & = 
	\left\{
	(x, y)\in \bbZ^2 \,| \, \dfrac{x+y}{\sqrt 2}  > \left(\dfrac{x-y}{\sqrt 2}\right)^2
	\right\}	,\\
	W_- & = 
	\left\{
	(x, y)\in \bbZ^2 \,| \, \dfrac{x+y}{\sqrt 2}  < \left(\dfrac{x-y}{\sqrt 2}\right)^2
	\right\}.
\end{align*}
None of the subsets $W_+, W_-$ admit any minimal complement in $\bbZ^2$ \textnormal{[see Fig.\ref{Fig5}]}. 
\end{corollary}

\begin{proof}
Let $G_1, G_2$ denote the subgroups $\{0\}\times \bbZ, \{
	(x, y)\in \bbZ^2 \,| \,  y = x
	\}$ of $\bbZ^2$ respectively. 
Note that $\bbZ^2$ is equal to $G_1G_2$, which we identify with $G_1\times G_2$ via the product map. Let $\varphi:G_2\xra{\sim} \bbZ$ denote the map given by 
$$
\varphi(x,y) = x, \quad (x, y)\in G_2,$$
$u:G_1\to G_2$ denote the map given by 
$$u(0,a) 
= 
\left(\left\lfloor \dfrac{a^2}{2\sqrt 2} - \dfrac a2 \right\rfloor + \left\lceil \dfrac{a^2}{2\sqrt 2} - \dfrac a2 - \left\lfloor  \dfrac{a^2}{2\sqrt 2} - \dfrac a2\right\rfloor \right\rceil, \left\lfloor \dfrac{a^2}{2\sqrt 2} - \dfrac a2 \right\rfloor + \left\lceil \dfrac{a^2}{2\sqrt 2} - \dfrac a2 - \left\lfloor  \dfrac{a^2}{2\sqrt 2} - \dfrac a2\right\rfloor \right\rceil\right),
$$
and $u':G_1\to G_2$ denote the map given by 
$$u'(0,a) 
= 
\left(\left\lfloor \dfrac{a^2}{2\sqrt 2} - \dfrac a2 \right\rfloor , \left\lfloor \dfrac{a^2}{2\sqrt 2} - \dfrac a2 \right\rfloor \right).
$$
Then it follows that 
\begin{align*}
W_+ & = \cup_{g\in G_1} \{g\}\times \varphi^\mo \left(\bbZ_{<\varphi(u(g))}\right),\\
W_- & = \cup_{g\in G_1} \{g\}\times \varphi^\mo \left(\bbZ_{>\varphi(u'(g))}\right).
\end{align*}
By Proposition \ref{Prop:Truncated30DSlant}, the corollary follows.
\end{proof}

\begin{figure}[h]
	\centering
	\begin{subfigure}{.5\textwidth}
		\centering
		\includegraphics[scale=0.08]{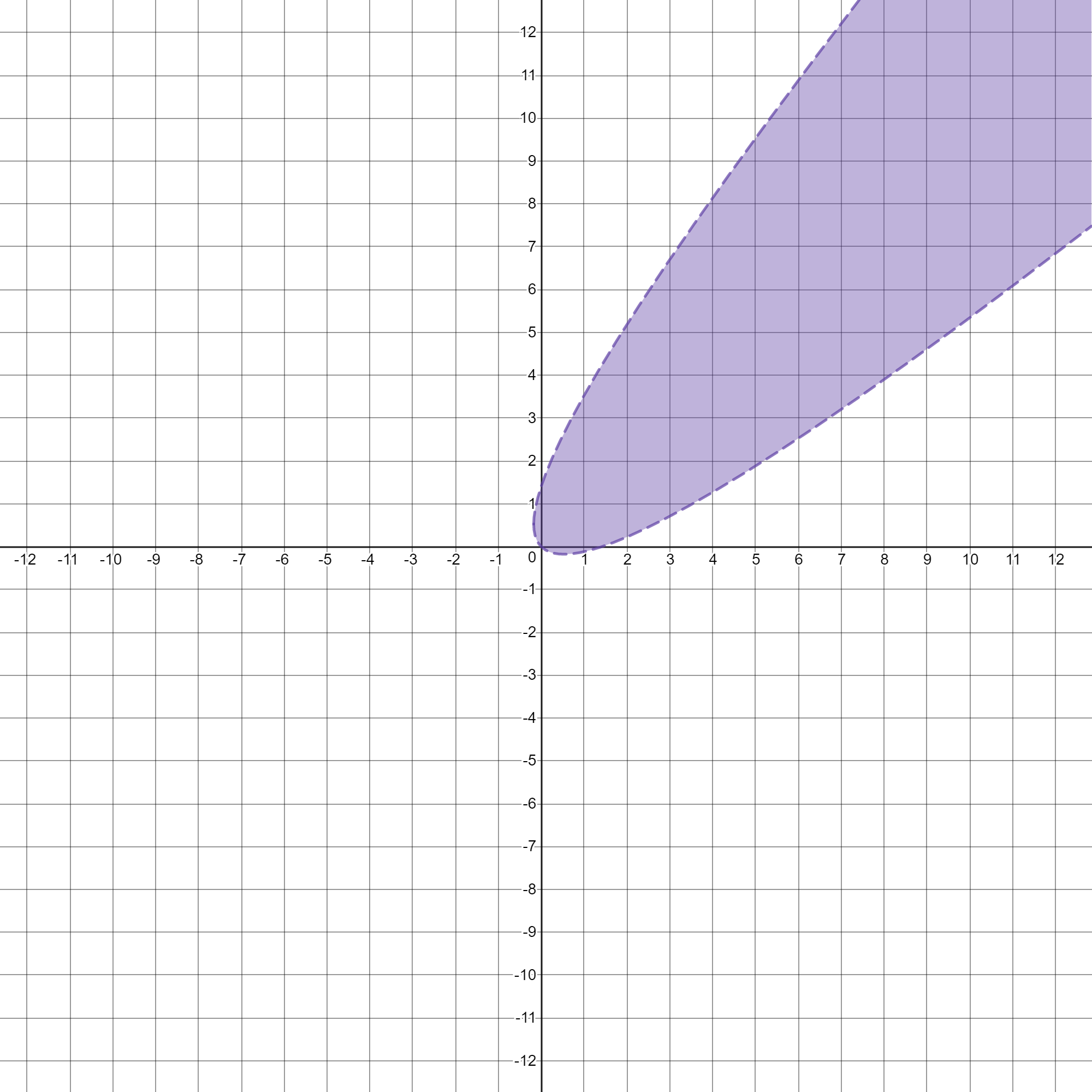}
		\caption{$W_{+}$}
		\label{}
	\end{subfigure}%
	\begin{subfigure}{.5\textwidth}
		\centering
		\includegraphics[scale=0.08]{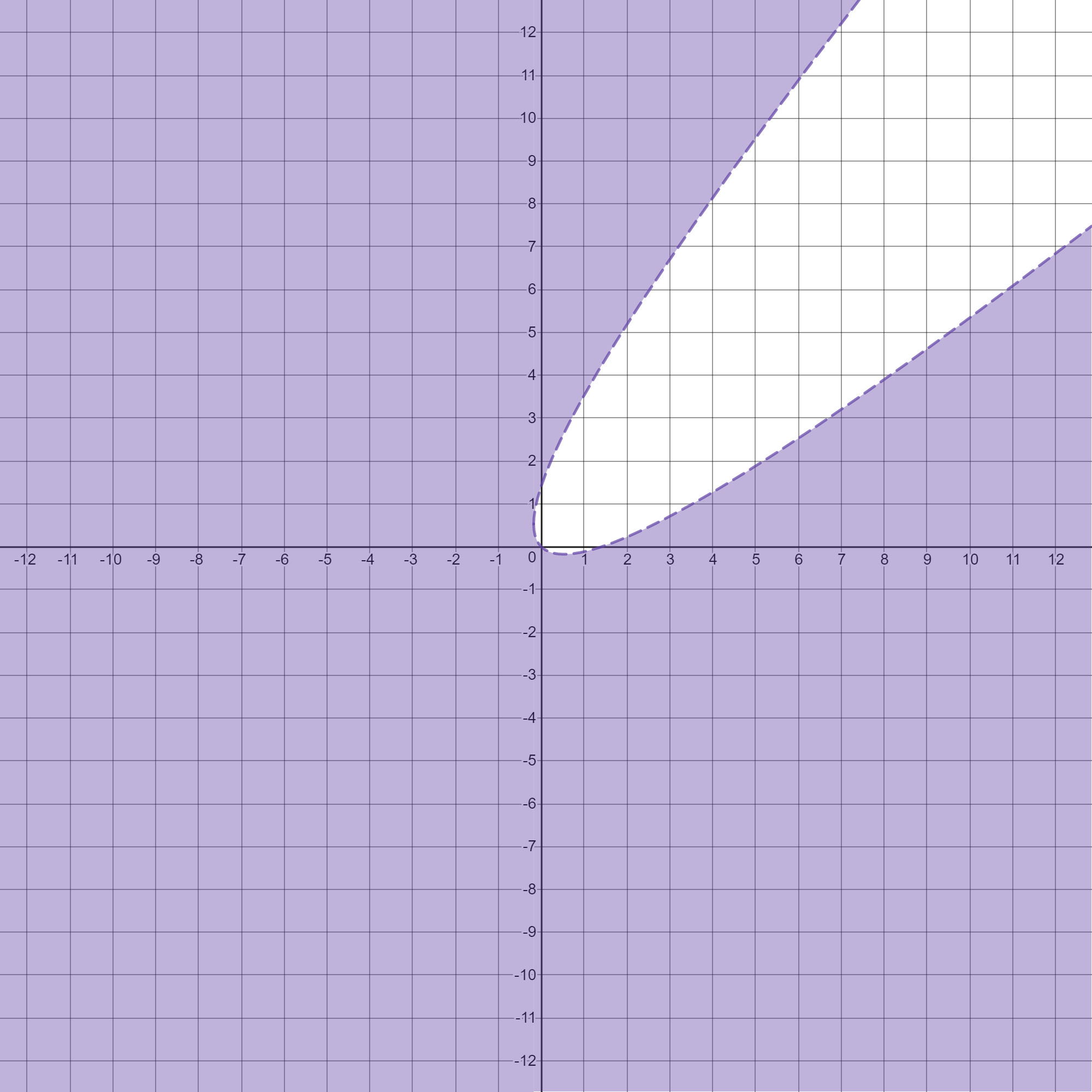}
		\caption{$W_{-}$}
		\label{}
	\end{subfigure}
	\caption{Cor. \ref{Cor:LowerParabolaRota}. $W_{+}$ and $W_{-}$ don't admit any minimal complement in $\mathbb{Z}^{2}$}
	\label{Fig5}
\end{figure}

\begin{corollary}
	\label{Cor:sideWhatever}
	Let $f:\bbZ \to \bbR$ be a function and let $W_\pm$ denote the subset of $\bbZ^2$ defined by
	\begin{align*}
	W_\pm & = 
	\{
	(x, y)\in \bbZ^2 \,| \, \pm(x-f(y)) > 0
	\}.
	\end{align*}
	None of the subsets $W_+, W_-$ admit any minimal complement in $\bbZ^2$ \textnormal{[see Fig.\ref{Fig6}]}. 
\end{corollary}

\begin{figure}[h]
	\centering
	\begin{subfigure}{.5\textwidth}
		\centering
		\includegraphics[scale=0.08]{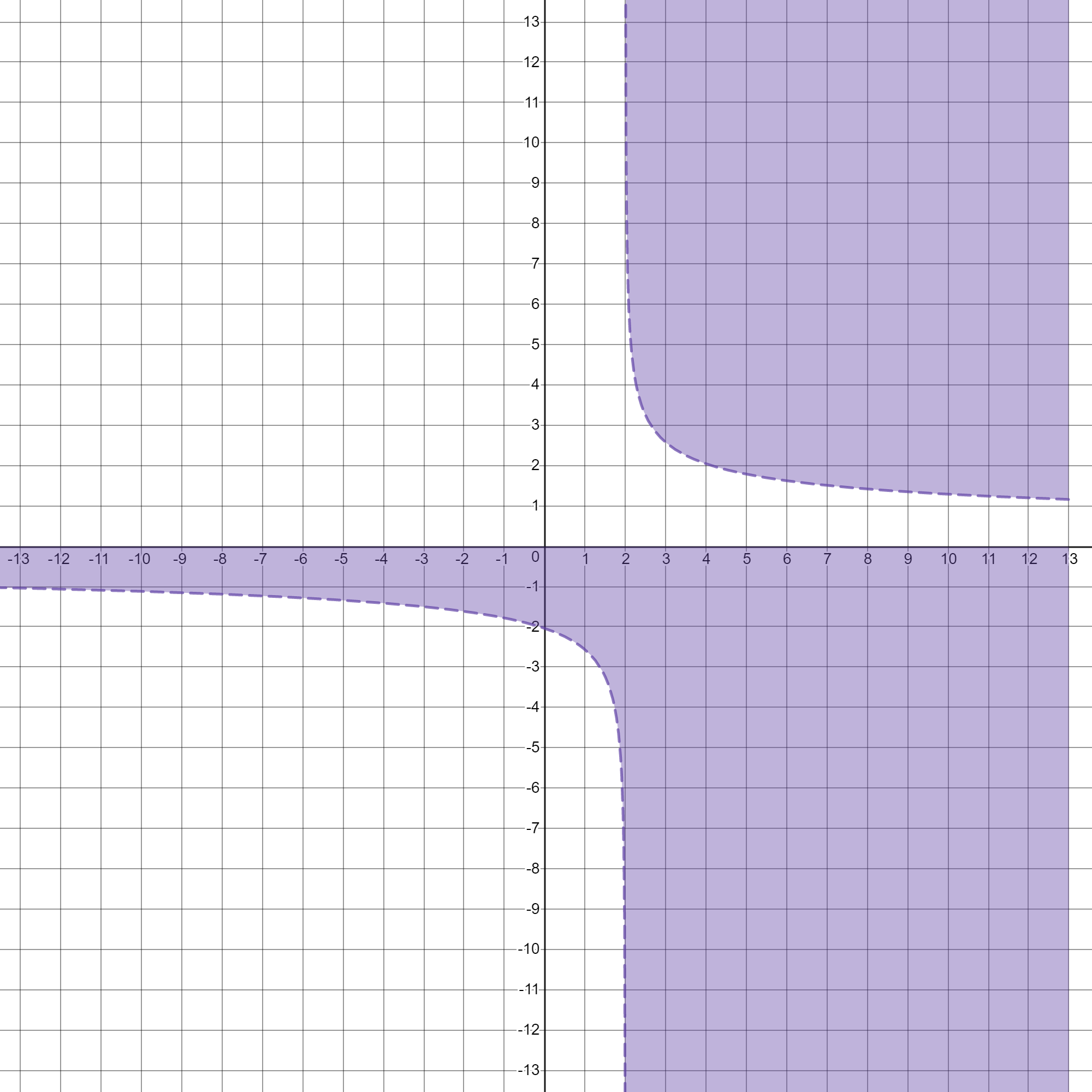}
		\caption{$W_{+}: f(y) = 17y^{-3}+2,y\neq 0, f(0)=t$}
		\label{}
	\end{subfigure}%
	\begin{subfigure}{.5\textwidth}
		\centering
		\includegraphics[scale=0.08]{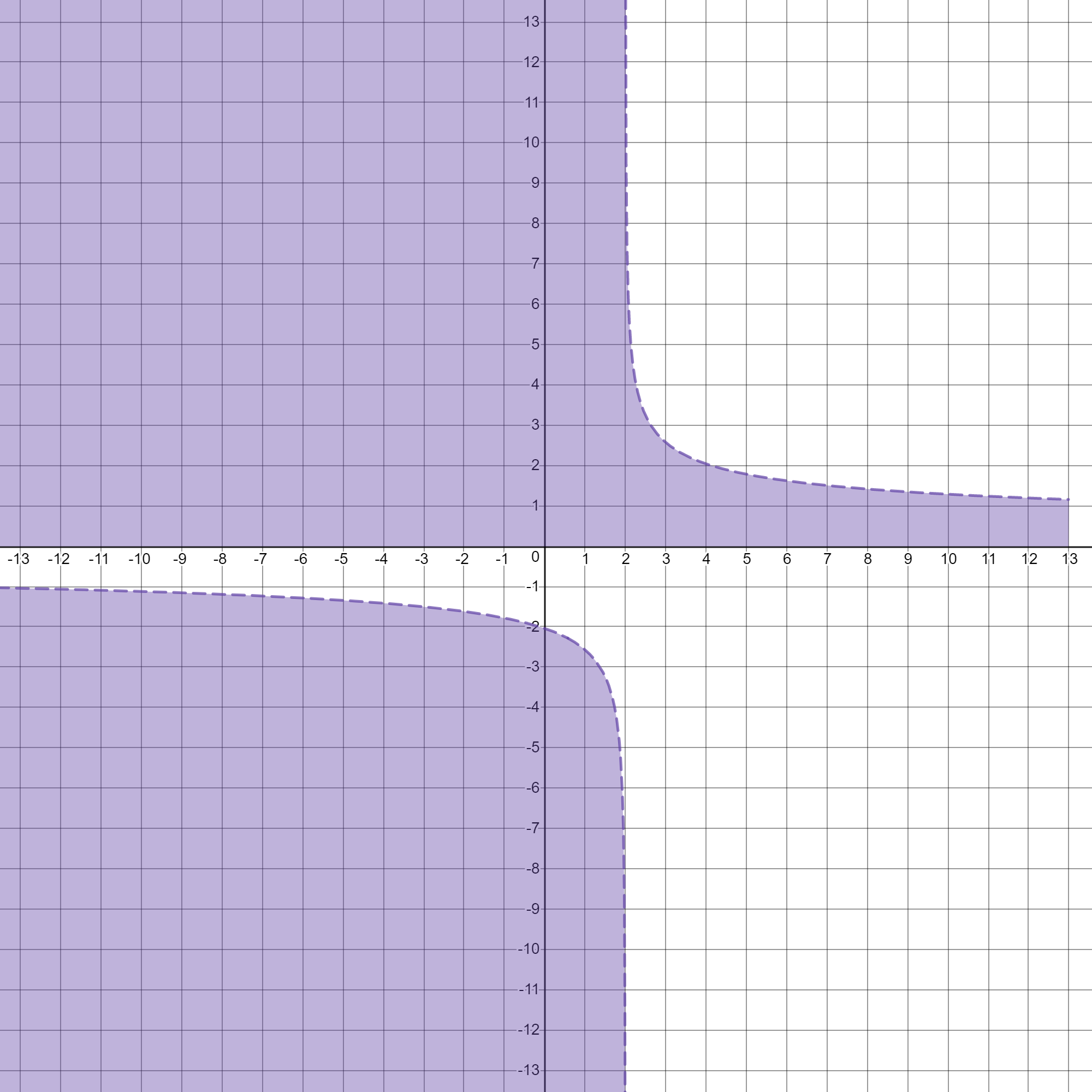}
		\caption{$W_{-}: f(y) = 17y^{-3}+2 ,y\neq 0,f(0)=t$}
		\label{}
	\end{subfigure}
	\caption{Cor. \ref{Cor:sideWhatever}. $W_{+}$ and $W_{-}$ with a parameter $t$. For any $t\in \mathbb{Z}$, they don't admit any minimal complement in $\mathbb{Z}^{2}$. In the figure, graph of $t=13$ is shown.}
	\label{Fig6}
\end{figure}

\begin{corollary}
\label{Cor:LowerWhateverRota}
Let $f: \bbZ \to \bbR$ be a function and let $W_+$ (resp.$W_-$) denote the subset of $\bbZ^2$ lying `above' (resp. `below') the graph of $Y=f(X)$ rotated clockwise by an angle $\theta$, i.e., 
\begin{align*}
	W_+ & = 
	\left\{
	(x, y)\in \bbZ^2 \,| \, x\sin \theta + y \cos \theta > f(x\cos \theta - y \sin\theta)
	\right\}	,\\
	W_- & = 
	\left\{
	(x, y)\in \bbZ^2 \,| \,x\sin \theta + y \cos \theta < f(x\cos \theta - y \sin\theta)	\right\}.
\end{align*}
If $\tan \theta$ is rational\footnote{$\tan \frac \pi 2$ is to be interpreted as ``$\frac 10$' and hence as a rational.}, then none of the subsets $W_+, W_-$ admit any minimal complement in $\bbZ^2$. 
\end{corollary}

\begin{proof}
Suppose $a,b$ are integers with no common factor and $\tan \theta = \frac ab$. Choose integers $c,d$ such that $ad-bc = -1$. Let $G_1, G_2$ denote the subgroups $\{
	(x, y)\in \bbZ^2 \,| \,  c x = dy
	\}, \{
	(x, y)\in \bbZ^2 \,| \,  a x = by
	\}$ of $\bbZ^2$ respectively. 
Note that $\bbZ^2$ is equal to $G_1G_2$, which we identify with $G_1\times G_2$ via the product map. Let $\varphi:G_2\xra{\sim} \bbZ$ denote the map given by 
$$
\varphi(x,y) = \frac xb, \quad (x, y)\in G_2,$$
$u:G_1\to G_2$, $u':G_1\to G_2$ denote the maps given by 
\begin{align*}
u(dt, ct) 
& = 
\left( b(\left\lfloor \alpha_t \right\rfloor + \left\lceil \alpha_t - \left\lfloor  \alpha_t\right\rfloor \right\rceil), a(\left\lfloor \alpha_t \right\rfloor + \left\lceil \alpha_t - \left\lfloor  \alpha_t\right\rfloor \right\rceil)\right),\\
u'(dt,ct) 
&= 
\left(b\left\lfloor \alpha_t \right\rfloor , a\left\lfloor \alpha_t \right\rfloor \right)
\end{align*}
where 
$$\alpha_t = 
\dfrac 1 {\sqrt{a^2 + b^2}} \left( 
f\left(\dfrac {-t}{\sqrt{a^2 + b^2}} \right)- \dfrac {ac+bd}{\sqrt {a^2 + b^2}} t
\right).
$$
Then it follows that 
\begin{align*}
W_+ & = \cup_{g\in G_1} \{g\}\times \varphi^\mo \left(\bbZ_{<\varphi(u(g))}\right),\\
W_- & = \cup_{g\in G_1} \{g\}\times \varphi^\mo \left(\bbZ_{>\varphi(u'(g))}\right).
\end{align*}
By Proposition \ref{Prop:Truncated30DSlant}, the corollary follows.
\end{proof}

In higher dimensions, the parabola is replaced by an $n$-dimensional paraboloid.
\begin{corollary}
\label{Cor:LowerParaboloid}
Let $n$ be a positive integer. Let $W_+$ (resp. $W_-$) denote the subset of $\bbZ^{n+1}$ consisting of points lying above (resp. below) the $n$-dimensional paraboloid $Z=X_1^2+X_2^2 + \cdots + X_n^2$, i.e.,
\begin{align*}
	W_+ & = 
	\{
	(x_1, \cdots, x_n, x_{n+1})\in \bbZ^{n+1} \,| \, + (x_{n+1} - x_1^2-x_2^2 - \cdots - x_n^2) >0
	\},\\
	W_- & = 
	\{
	(x_1, \cdots, x_n, x_{n+1})\in \bbZ^{n+1} \,| \, - (x_{n+1} - x_1^2-x_2^2 - \cdots - x_n^2) >0
	\}.
\end{align*}
None of the subsets $W_+, W_-$ admit any minimal complement in $\bbZ^{n+1}$. 
\end{corollary}

More generally, we have the following result. 

\begin{corollary}
\label{Cor:LowerWhatever3D}
Let $n$ be a positive integer. 
Let $f:\bbZ^n \to \bbR$ be a function and let $W_\pm$ denote the subset 
$$W_\pm 
=\{
(x_1, \cdots , x_n, x_{n+1})\in \bbZ^{n+1}\,|\, \pm (x_{n+1} -f(x_1, \cdots, x_n))>0 
\}
$$
of $\bbZ^{n+1}$. 
None of the subsets $W_+, W_-$ admit any minimal complement in $\bbZ^{n+1}$. 
\end{corollary}

\begin{proof}
Note that $W_+$ is the union of the truncated lines $$\{(m_1, \cdots, m_n)\} \times (\lfloor f(m_1, \cdots, m_n)\rfloor , \infty)$$ as $(m_1, \cdots, m_n)$ varies in $\bbZ^n$, and $W_-$ is the union of the truncated lines $$\{(m_1, \cdots, m_n)\} \times (-\infty, \lfloor f(m_1, \cdots, m_n)\rfloor + \lceil f(m_1, \cdots, m_n)- \lfloor f(m_1, \cdots, m_n)\rfloor \rceil)$$ as $(m_1, \cdots, m_n)$ varies in $\bbZ^n$. So by Proposition \ref{Prop:Truncated30DSlant}, it follows that none of $W_+, W_-$ admit any minimal complement in $\bbZ^{n+1}$.
\end{proof}

\section{Subsets containing lines (spiked subsets)}
In the following, $k$ denotes a positive integer. 

\begin{definition}
[\textbf{Moderation}]
\label{moderation}
Let $u:\bbZ^k \to \bbZ$ be a function. A function $v:\bbZ^k\to \bbZ$ is said to be a \textnormal{moderation of} $u$ if for each $x_0\in \bbZ^k$, the function 
$$x\mapsto u(x) + v(x_0-x)$$
 defined on $\bbZ^k$ is bounded above. 
\end{definition}

For each $x\in \bbZ^k$ and for each positive integer $n$, let $B (x,n)$ denote the set of elements $y$ in $\bbZ^k$ such that $||x-y||^2< n$ holds. 

\begin{definition}
[\textbf{Spiked subsets}]
\label{Spiked subsets}
A subset $X$ of $\bbZ^{k+1}$ is said to be a \textnormal{spiked subset} if it satisfies 
$$\calB\times \bbZ
\subseteq
X
\subseteq 
(\calB\times \bbZ )
\bigsqcup 
\left(
\sqcup 
_{x\in \bbZ^k\setminus \calB}
\left(
\{
x
\}
\times 
(-\infty, u(x))
\right)
\right)$$
for some nonempty subset $\calB$ of $\bbZ^k$ and some function $u:\bbZ^k\to \bbZ$. The set $\calB$ is called the \textnormal{base} of $X$.
\end{definition}

\begin{proposition}[Spiked vs eventually periodic]
	\label{Spiked vs eventually periodic}
For each integer $d\geq 2$, no eventually periodic subset of $\bbZ^d$ is spiked, and no spiked subset of $\bbZ^d$ is eventually periodic. 
\end{proposition}

\begin{proof}
Suppose there exists an eventually periodic subset $W$ of $\bbZ^d$ for some integer $d\geq 2$ such that $W$ is spiked. So certain translate $W'$ of $W$ contains a subgroup of $\bbZ^d$ of rank one and hence contains the inverse of infinitely many of its elements. Note that $W'$ is also an eventually periodic subset of $\bbZ^d$, and no eventually periodic subset contains the inverses of infinitely many of its elements. This proves the first statement. Then the second statement follows from the first statement. 
\end{proof}

\begin{proposition}
\label{Prop:LinesNoise}
Let $u:\bbZ^k\to \bbZ$ be a function. Suppose $v: \bbZ^k\to \bbZ$ be a moderation $u$. Let $\calB$ be a subset of $\bbZ^k$ admitting a minimal complement $M$ in $\bbZ^k$. Then the subset 
$$M_v:=\{(x, v(x))\,|\, x\in M\}$$
of $\bbZ^{k+1}$ is a minimal complement of any subset $X$ of $\bbZ^{k+1}$ satisfying 
$$\calB\times \bbZ
\subseteq
X
\subseteq 
(\calB\times \bbZ )
\bigsqcup 
\left(
\sqcup 
_{x\in \bbZ^k\setminus \calB}
\left(
\{
x
\}
\times 
(-\infty, u(x))
\right)
\right).
$$
\end{proposition}

\begin{proof}
For any $x\in \bbZ^k$, we have $$(\calB\times \bbZ) + (x, v(x)) = (\calB+x)\times \bbZ.$$ Since $$\calB+M = \bbZ^k,$$ it follows that  $$(\calB\times \bbZ) + M_v= \bbZ^{k+1}.$$ So $M_v$ is a complement of $\calB\times \bbZ$ in $\bbZ^{k+1}$. Since $X$ contains $\calB\times \bbZ$, the set $M_v$ is also a complement of $X$. 

Suppose $M_v$ is not a minimal complement of $X$ in $\bbZ^{k+1}$. Then for some $x'\in M$, the set $M_v\setminus \{(x', v(x'))\}$ is also a complement of $X$. Since $M$ is a minimal complement to $\calB$, it follows that $$x_0\notin\calB+ (M\setminus \{x'\})$$ for some element $x_0\in \bbZ^k$. So $$\Big( (\calB\times \bbZ) + (M_v\setminus \{(x', v(x'))\})\Big)\cap \Big(\{x_0\}\times \bbZ\Big) = \emptyset.$$

Since $v$ is a moderation of $u$, it follows that there exists an integer $m_0$ such that $$u(x)+v(x_0 - x)\leq m_0 \text{ for any }x\in \bbZ^k.$$
Note that if $$\left(
\{
x
\}
\times 
(-\infty, u(x)) \right)
+(y, v(y)) \subseteq\{x_0\} \times \bbZ$$ for some elements $x,y\in \bbZ^k$, then $y$ is equal to $x_0-x$ and hence 
$$\left(
\{
x
\}
\times 
(-\infty, u(x)) \right)
+(y, v(y))
\subseteq \{x_0\} \times (-\infty, m_0).
$$ So
$$\left( \left(
\{
x
\}
\times 
(-\infty, u(x)) \right)
+M_v \right) 
\cap 
\left(
\{x_0\}\times \bbZ\right) \subseteq\{x_0\} \times (-\infty, m_0).$$ Consequently, 
$$\left( 
\left(
\cup _{x\in \bbZ^k}
\left(
\{
x
\}
\times 
(-\infty, u(x))
\right)
\right)
+M_v \right) 
\cap 
\left(
\{x_0\}\times \bbZ
\right)
\subseteq\{x_0\} \times (-\infty, m_0).$$ This contradicts the assumption that $M_v$ is not a minimal complement of $X$ in $\bbZ^{k+1}$. Hence the result follows. 
\end{proof}

\begin{lemma}
\label{Lemma:ModerationExists}
Any function $u:\bbZ^k\to \bbZ$ admits a moderation. 
\end{lemma}

\begin{proof}
Define $v:\bbZ^k\to \bbZ$ by 
$$v(x)= -\max
\left(u\left(
B(-x, ||x||^2+1)
\right)\right), 
\quad 
\text{ for } x\in \bbZ^k.
$$
Let $x_0$ be an element of $\bbZ^k$. Then for any $x\in \bbZ^k$ satisfying $||x||^2\geq ||x_0||^2$, $$x \in B(x-x_0, ||x||^2+1).$$ Hence 
$$
u(x) \leq \max
\left(u\left(
B(x-x_0, ||x||^2+1)
\right)\right) = -v(x_0-x).
$$
Since $||x||^2<||x_0||^2$ holds only for finitely many elements of $x\in \bbZ^k$, it follows that $v$ is a moderation of $u$. This proves the result.
\end{proof}

\begin{theorem}
\label{Thm:LinesNoise}
Let $u:\bbZ^k\to \bbZ$ be a function. Let $\calB$ be a subset of $\bbZ^k$ admitting a minimal complement $M$ in $\bbZ^k$. Then $u$ admits a moderation and for any moderation $v$ of $u$, the subset 
$$M_v:=\{(x, v(x))\,|\, x\in M\}$$
of $\bbZ^{k+1}$ is a minimal complement of any subset $X$ of $\bbZ^{k+1}$ satisfying 
\begin{equation}
\label{Eqn:ContainmentX}
\calB\times \bbZ
\subseteq
X
\subseteq 
(\calB\times \bbZ )
\bigsqcup 
\left(
\sqcup 
_{x\in \bbZ^k\setminus \calB}
\left(
\{
x
\}
\times 
(-\infty, u(x))
\right)
\right).
\end{equation}
\end{theorem}

\begin{proof}
It follows from Proposition \ref{Prop:LinesNoise} and Lemma \ref{Lemma:ModerationExists}.
\end{proof}

\begin{remark}
Note that $\calB\times \bbZ$ is the only subset of $\bbZ^{k+1}$ which satisfies Equation \eqref{Eqn:ContainmentX} and is equal to the cartesian product of a subset of $\bbZ^a$ and a subset of $\bbZ^b$ for some positive integers $a,b$ with $a+b = k+1$. Thus the subsets dealt in Theorem \ref{Thm:LinesNoise} (other than $\calB\times \bbZ$) are different from what has been considered in \cite[\S 5]{MinComp1}, and hence does not fall under the purview of \cite[\S 5]{MinComp1}.
\end{remark}

\begin{corollary}\label{Cor4.8}
Let $W$ denote the subset of $\bbZ^2$ consisting of points lying below the parabola $Y=X^2$ and points lying on the $y$-axis, i.e.,
$$W = 
\{
(x, y)\in \bbZ^2 \,| \, y - x^2 <0
\}
\cup (\{0\}\times \bbZ).$$
The set 
$$M=
\{
(t, -2t^2)\in \bbZ^2 \,| \, t\in \bbZ
\}$$
is a minimal complement of $W$ in $\bbZ^2$  \textnormal{[see Fig.\ref{Fig7}]}.
\end{corollary}

\begin{proof}
Let $u,v:\bbZ \to \bbZ$ denote the functions defined by $$u(t) = t^2, v(t) = -2t^2, \,\,t\in \bbZ.$$ Then for any fixed $t_0\in \bbZ$, it follows that 
$$
u(t) + v (t_0-t) 
= t^2 - 2 (t_0-t)^2
= -2t_0^2 + 4tt_0 - t^2
\leq 2t_0^2
$$
for any $t\in \bbZ$. Hence $v$ is a moderation of $u$. Since $\bbZ$ is a minimal complement of $\{0\}$ in $\bbZ$, it follows from Theorem \ref{Thm:LinesNoise} that $M$ is a minimal complement of $W$.
\end{proof}

\begin{figure}[h]
	\centering
	\begin{subfigure}{.5\textwidth}
		\centering
		\includegraphics[scale=0.08]{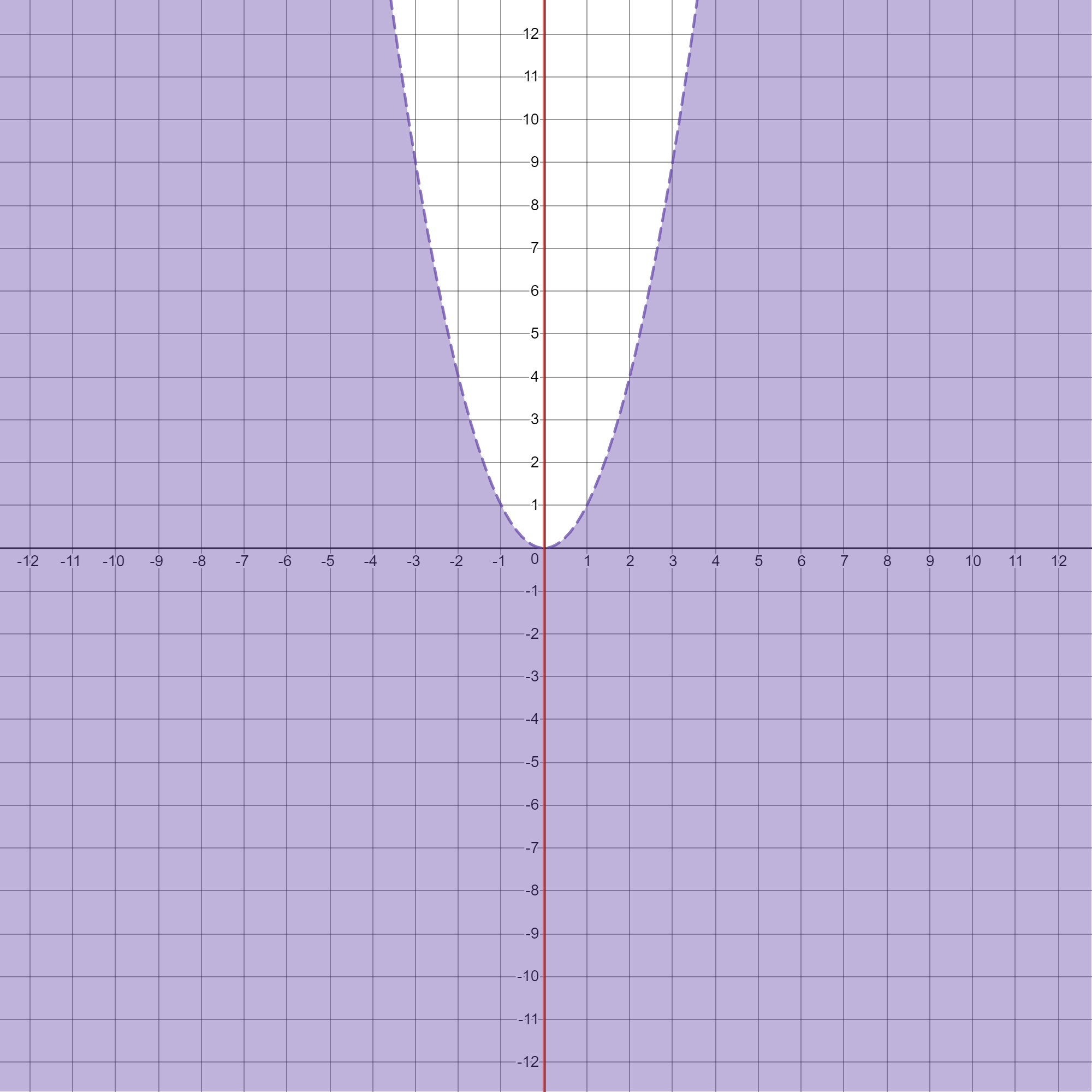}
		\caption{$W$}
		\label{}
	\end{subfigure}%
	\begin{subfigure}{.5\textwidth}
	\centering
	\includegraphics[scale=0.08]{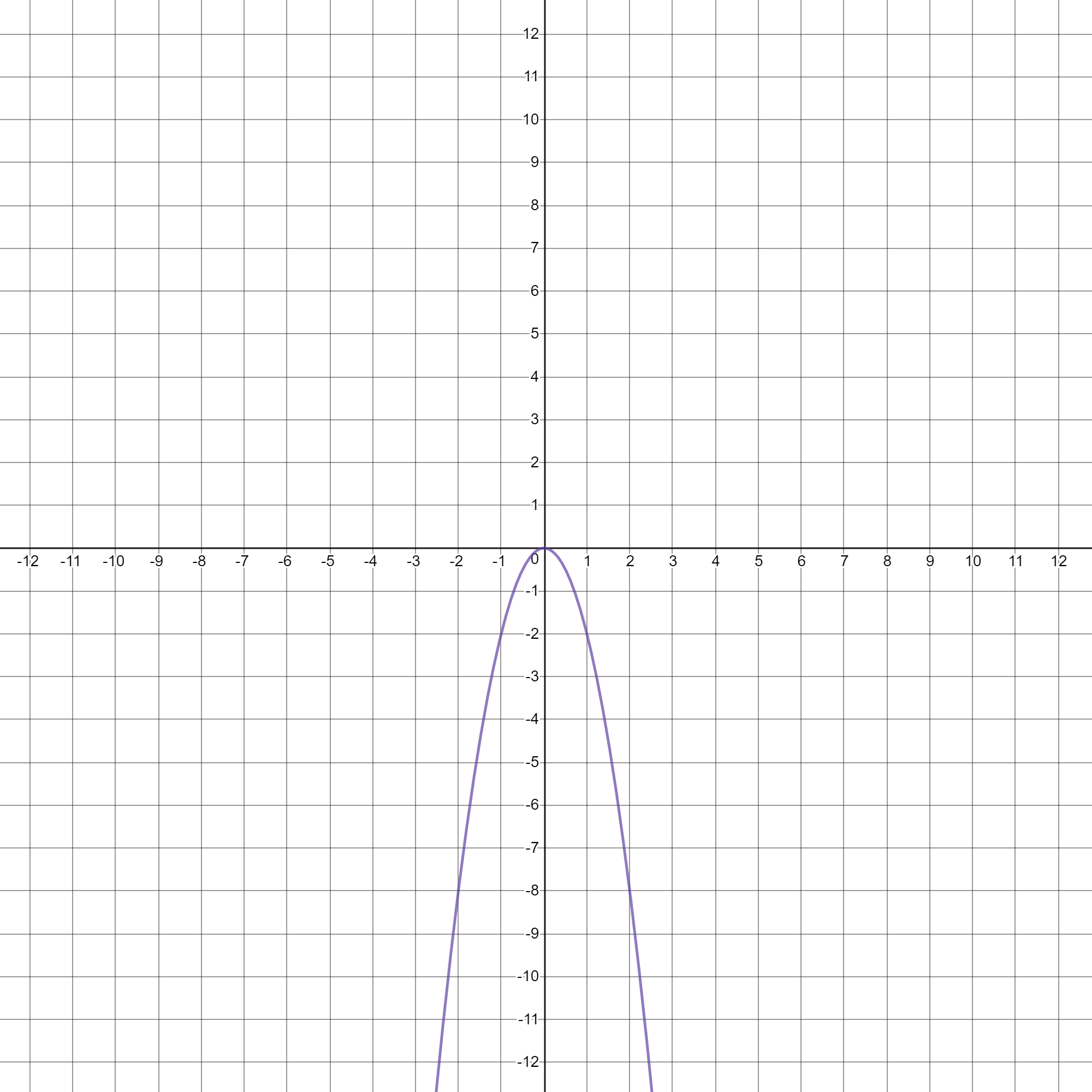}
	\caption{$M$}
	\label{}
	\end{subfigure}
	\caption{Cor. \ref{Cor4.8}. $M$ is a minimal complement to $W$ in $\mathbb{Z}^{2}$}
	\label{Fig7}
\end{figure}

\begin{corollary}\label{Cor4.9}
Let $n$ be a positive integer. Let $W$ denote the subset of $\bbZ^{n+1}$ consisting of points lying below the $n$-dimensional paraboloid $X_{n+1}=X_1^2+\cdots+ X_n^2$ and points lying on the line $X_1 = \cdots = X_n = 0$, i.e.,
$$W = 
\{
(x_1,\cdots, x_n, x_{n+1})\in \bbZ^{n+1} \,| \, x_{n+1} - x_1^2 - \cdots - x_n^2<0
\}
\cup (\{(0,\cdots, 0)\}\times \bbZ).$$
The set 
$$M=
\{
(s_1,\cdots, s_n, -2(s_1^2+\cdots + s_n^2)) \,| \, (s_1,\cdots, s_n)\in \bbZ^n
\}$$
is a minimal complement of $W$ in $\bbZ^{n+1}$. More generally, for any nonempty finite subset $A$ of $\bbZ^n$, the set 
$$\{
(x_1,\cdots, x_n, x_{n+1})\in \bbZ^{n+1} \,| \, x_{n+1} - x_1^2 - \cdots - x_n^2<0
\}
\cup (A\times \bbZ)$$
admits a minimal complement in $\bbZ^{n+1}$. 
\end{corollary}

\begin{proof}
Let $u,v:\bbZ^n \to \bbZ$ denote the functions defined by 
$$u(s_1,\cdots, s_n) = s_1^2 + \cdots + s_n^2, v(s_1,\cdots, s_n) = -2(s_1^2 +\cdots +  s_n^2)$$ for $(s_1,\cdots, s_n)\in \bbZ^n$. Then for any fixed $(s_1',\cdots, s_n')\in \bbZ^n$, it follows that 
$$u(s_1,\cdots, s_n) + v ((s_1',\cdots, s_n')-(s_1,\cdots, s_n))=\sum_{i=1}^n (-2s_i'^2 + 4s_is_i' -s_i^2 )\leq \sum_{i=1}^n 2s_i'^2$$ 
for any $(s_1,\cdots, s_n)\in \bbZ^n$. Hence $v$ is a moderation of $u$. Since $\bbZ^n$ is a minimal complement of $\{(0,\cdots, 0)\}$ in $\bbZ^n$, it follows from Theorem \ref{Thm:LinesNoise} that $M$ is a minimal complement of $W$.

By \cite[Theorem 2.1]{MinComp1}, any nonempty finite subset $A$ of $\bbZ^n$ admits a minimal complement $B$ in $\bbZ^n$. From Theorem \ref{Thm:LinesNoise}, it follows that 
$$M=
\{
(s_1,\cdots, s_n, -2(s_1^2+\cdots + s_n^2)) \,| \, (s_1,\cdots, s_n)\in B
\}$$
is a minimal complement of $$\{
(x_1,\cdots, x_n, x_{n+1})\in \bbZ^{n+1} \,| \, x_{n+1} - x_1^2 - \cdots - x_n^2<0
\}
\cup (A\times \bbZ).$$
\end{proof}

\begin{remark}
\label{Rk:PolyModeration}
Note that if $u$ is a polynomial function, then $v$ can also be taken to be some appropriate polynomial. Indeed, if $x_1, \cdots, x_n$ are integers and $i_1, \cdots, i_n$ are non-negative integers, then the inequalities 
\begin{align*}
x_1^{i_1} \cdots x_n^{i_n}
& \leq |x_1^{i_1}| \cdots |x_n^{i_n}| \\
& \leq (|x_1|\cdots |x_n|)^{\sum i_j} \\
& \leq \frac 1n \left(|x_1|^{n\sum i_j} + \cdots + |x_n|^{n\sum i_j}\right) \\
& \leq \frac 1n \left(x_1^{2n\sum i_j} + \cdots + x_n^{2n\sum i_j}\right)
\end{align*}
hold. 
Hence given a nonzero polynomial $u \in \bbZ[X_1, \cdots, X_n]$ of degree $d$, there exists a positive integer $k$ such that the function $u:\bbZ^n\to \bbZ$ is less than or equal to the polynomial function $p: \mathbb{Z}^{n}\to \mathbb{Z}$ defined by,
$$ p(X_1,\cdots, X_n) = k(X_1^{2nd} + \cdots + X_n^{2nd}).$$
Then the function $v:\bbZ^n\to \bbZ$ defined by the polynomial 
$$v(X_1,\cdots, X_n)= -2k(X_1^{2nd} + \cdots + X_n^{2nd})$$ is a moderation of $u$. 
\end{remark}

\section{Spiked subsets in arbitrary abelian groups}

In the following, $\calG$ denotes an abelian group. The $k$-fold product $\bbZ^k$ of $\bbZ$ is assumed to be equipped with the dictionary order induced by the order of $\bbZ$.

\begin{definition}[Spiked sets in abelian groups]
A subset $X$ of $\calG$ is said to \textnormal{contain spikes} if there exist subgroups $G_1, G_2$ of $\calG$ such that 
\begin{enumerate}
\item $G_2$ is free of positive rank,
\item $G_1\cap G_2$ is trivial, 
\item $X$ is contained in $G_1G_2$,
\item $X$ contains $g_1 G_2$ for some $g_1\in G_1$.
\end{enumerate}
In this situation, we say that $X$ is \textnormal{spiked} with respect to $G_1, G_2$, and the set 
$$\calB = \{g_1\in G_1 \,|\, g_1G_2\subseteq X\}$$
is called the \textnormal{base} of the spiked set $X$ with respect to $G_1, G_2$.
\end{definition}

The groups $G_1, G_2$ as above have trivial intersection and hence the map 
$$G_1\times G_2\xra{(g_1, g_2)\mapsto g_1g_2} G_1G_2$$ defined by taking products  is an isomorphism. Henceforth we identify the groups $G_1\times G_2, G_1G_2$ via this isomorphism.

\begin{definition}[Bounded spiked set]
A spiked subset $X$ of $\calG$ with respect to $G_1, G_2$ is called $(u, \varphi)$-\textnormal{bounded} if there exists a function $u:G_1\to G_2$ and an isomorphism $$\varphi:G_2\xra{\sim} \bbZ^{\rk G_2}$$ such that 
$$\calB  G_2
\subseteq
X
\subseteq 
\calB G_2
\bigsqcup 
\left(
\bigsqcup_{g_1\in G_1\setminus \calB}
g_1\cdot
\left(
\varphi^\mo \left(\bbZ^{\rk G_2}_{<\varphi(u(g_1))}
\right)
\right)
\right).
$$
\end{definition}

The main guiding example is $\calG = \bbZ^{k_1+k_2}$, $G_1 = \bbZ^{k_1}$, $G_2 = \bbZ^{k_2}, \varphi = \mathrm{id}$.

\begin{definition}[moderation in abelian groups]
Let $\calF_1$ be an abelian group, $\calF_2$ be a free abelian group of positive rank and $u: \calF_1\to \calF_2$ be a function. Given an isomorphism $$\varphi: \calF_2 \to \bbZ ^{\rk \calF_2},$$ we say that a function $v: \calF_1 \to \calF_2$ is a $\varphi$-\textnormal{moderation} of $u$ if for each $x_0\in \calF_1$, the function $$x\mapsto \varphi(v(x) + u(x_0-x))$$ defined on $\calF_1$ is bounded above, i.e., the function $$(a,b)\mapsto \varphi( u(a) + v(b))$$ is bounded above on each fibre of the map $$\calF_1\times \calF_1\xra{(x, x')\mapsto x+x' } \calF_1.$$
\end{definition}

\begin{proposition}
\label{Prop:ModerationExists}
Let $\calF_1$ be a finitely generated abelian group, $\calF_2$ be a free abelian group of positive rank. For any isomorphism 
$$\varphi: \calF_2\xra{\sim} \bbZ^{\rk \calF_2},$$ any function $u: \calF_1 \to \calF_2$ admits a $\varphi$-moderation $v:\calF_1 \to \calF_2$. Moreover, for any isomorphism 
$$\varphi: \calF_2\xra{\sim} \bbZ^{\rk \calF_2}$$ and any subgroup $\calF_2'$ of $\calF_2$ of finite index, any function $u: \calF_1 \to \calF_2$ admits a $\varphi$-moderation $v':\calF_1 \to \calF_2$ such that $v'$ has values in $\calF_2'$. 
\end{proposition}

\begin{proof}
Note that any finitely generated abelian group can be equipped with a metric such that any ball contains only finitely many elements. Indeed, by the structure theorem of finitely generated abelian groups, such a group $G$ can be thought of as a subgroup of the product of a free abelian group of finite rank and finitely many copies of the unit circle. This induces a metric on $G$ such that any ball at any point with any radius contains only finitely many elements. 

For each positive integer $n$ and $x\in \calF_1$, let 
$$B (x,n) := \lbrace y\in \calF_1 : d(x,y)< n\rbrace.$$ For each $x\in \calF_1$, choose an element of $y_x$ of $\calF_2$ such that $\varphi(y_x)$
is greater than or equal to any element of the nonempty finite set 
$$\varphi
\left(u\left(
B(-x, d(x,0)+1)
\right)\right).$$ 
Define a map $v :\calF_1 \to \calF_2$ by 
$$v(x) = -y_x \quad \text{ for } x\in \calF_1.$$
Let $x_0$ be an element of $\calF_1$. Then for any $x\in \calF_1$ satisfying $d(x,0) \geq d(x_0,0)$, the point $x_0-x$ belongs to $B(-x, d(x,0)+1)$, and hence 
$$
\varphi(u(x_0-x)) \leq 
\varphi(y_x)
= \varphi(-v(x)),
$$
which gives $$\varphi(v(x) + u(x_0-x))\leq 0.$$ 
Since $d(x,0)<d(x_0,0)$ holds only for finitely many elements of $x\in \calF_1$, it follows that $v$ is a $\varphi$-moderation of $u$. This proves the existence of a moderation of $u$.

To prove the second part, choose a finite set $C$ of coset representative of $\calF_2'$ in $\calF_2$. 
For each $x\in \calF_1$, let $c_x$ denote the unique element of $C$ such that $v(x)-c_x$ lies in $\calF_2'$. Define $v':\calF_1 \to \calF_2$ by 
$$v'(x) = v(x)-c_x
\quad 
\text{ for } 
x\in \calF_1.$$
Since $C$ is a finite set, it follows that $v'$ is a $\calF_2'$-valued $\varphi$-moderation of $u$. 
\end{proof}

\begin{lemma}
\label{Lemma:ModerationImplication}
Let $\calF_1$ be an abelian group, $\calF_2$ be a free abelian group of positive rank and $u:\calF_1 \to \calF_2$ be a function. Let $v:\calF_1\to \calF_2$ be a variation of $u$ with respect to an isomorphism $$\varphi: \calF_2\xra{\sim} \bbZ^{\rk \calF_2}.$$
Then 
$$
\{(x, v(x))\,|\, x\in \calF_1\} + 
\left(
\bigsqcup_{g_1\in \calF_1}
g_1\cdot
\left(
\varphi^\mo \left(\bbZ^{\rk \calF_2}_{<\varphi(u(g_1))}
\right)
\right)
\right)
$$
contains $\{x_0\}\times \calF_2$ for no $x_0\in \calF_1$. 
\end{lemma}

\begin{proof}
Since $v$ is a $\varphi$-moderation of $u$, it follows that there exists an element $m_0\in \bbZ^{\rk \calF_2}$ such that $$\varphi(u(x)+v(x_0 - x))\leq m_0$$ for any $x\in \calF_1$. 
Note that if $$\left(
\{
x
\}
\times 
\left(\varphi^\mo \left(\bbZ^{\rk \calF_2}_{< u(x)}\right) \right)
 \right)
+(y, v(y))$$ is contained in $\{x_0\} \times \calF_2$ for some elements $x,y\in \calF_1$, then $y$ is equal to $x_0-x$ and hence 
$$\left(
\{
x
\}
\times 
\left(\varphi^\mo \left(\bbZ^{\rk \calF_2}_{< u(x)}\right) \right)
\right)
+(y, v(y))\subseteq \{x_0\} \times \left(\varphi^\mo \left(\bbZ^{\rk \calF_2}_{< m_0}\right) \right).$$ So
$$\left( \left(
\{
x
\}
\times 
\left(\varphi^\mo \left(\bbZ^{\rk \calF_2}_{< u(x)}\right) \right)
\right)
+ 
\{(v, v(x)\,|\, x\in \calF_1\} 
\right) 
\cap 
\left(
\{x_0\}\times \calF_2\right) \subseteq \{x_0\} \times \left(\varphi^\mo \left(\bbZ^{\rk \calF_2}_{< m_0}\right) \right).$$ Consequently, 
$$\left( 
\left(
\cup _{x\in \calF_1}
\left(
\{
x
\}
\times 
\left(\varphi^\mo \left(\bbZ^{\rk \calF_2}_{< u(x)}\right) \right)
\right)
\right)
+\{(v, v(x))\,|\, x\in \calF_1\}  
\right) 
\cap 
\left(
\{x_0\}\times \calF_2
\right)$$ is contained in $$\{x_0\} \times \left(\varphi^\mo \left(\bbZ^{\rk \calF_2}_{< m_0}\right) \right).$$ This proves the lemma. 
\end{proof}

\begin{theorem}[Necessary and sufficient condition]
\label{Thm:NecessarySuffi}
Let $G_1, G_2$ be subgroups of $\calG$ such that the intersection $G_1\cap G_2$ is trivial and $G_2$ is free of positive rank. Let $\calB$ be a nonempty subset of $G_1$. Then the following statements are equivalent. 
\begin{enumerate}
\item $\calB$ admits a minimal complement in $\calG$.
\item $\calB$ admits a minimal complement in $G_1$.
\item Any $(u, \varphi)$-bounded spiked subset of $\calG$ with respect to $G_1, G_2$, having $\calB$ as its base, admits a minimal complement in $G_1\times G_2$ which is the graph of any $\varphi$-moderation of $u$ restricted to some subset of $G_1$. 
\item Some $(u, \varphi)$-bounded spiked subset of $\calG$ with respect to $G_1, G_2$, having $\calB$ as its base, admits a minimal complement in $G_1\times G_2$ which is the graph of any $\varphi$-moderation of $u$ restricted to some subset of $G_1$. 
\end{enumerate}
\end{theorem}

\begin{proof}
The first two statements are equivalent by Theorem \ref{Thm:MinCompInSubgp}.

Suppose $\calB$ admits a minimal complement $M$ in $G_1$. Let $X$ be a $(u, \varphi)$-bounded spiked subset of $\calG$ with respect to $G_1, G_2$, having $\calB$ as its base. By Proposition \ref{Prop:ModerationExists}, there exists a $\varphi$-moderation $v:G_1\to G_2$ of $u$. Let $M_v$ denote the graph of the restriction of $\varphi$ to $M$. For any $x\in G_1$,
$$(\calB\times G_2) + (x, v(x)) = (\calB+x)\times G_2.$$ Since $\calB+M = G_1$, it follows that $$(\calB \times G_2) + M_v = G_1\times G_2.$$ So $M_v$ is a complement of $\calB \times G_2$ in $G_1 \times G_2$. Since $X$ contains $\calB\times G_2$, the set $M_v$ is also a complement of $X$ in $G_1\times G_2$. 

Suppose $M_v$ is not a minimal complement of $X$ in $G_1\times G_2$. Then for some $x'\in M$, the set $M_v\setminus \{(x', v(x'))\}$ is also a complement to $X$. Since $M$ is a minimal complement to $\calB$ in $G_1$, it follows that $\calB+ (M\setminus \{x'\})$ does not contain $x_0$ for some element $x_0\in G_1$. So $$(\calB\times G_2) + (M_v\setminus \{(x', v(x'))\})$$ contains no element of $\{x_0\}\times G_2$. 
Since $(M_v\setminus \{(x', v(x'))\}$ is complement to $X$, it follows that $$\{x_0\}\times G_2\subseteq (\calB\times G_2) + (M_v\setminus \{(x', v(x'))\}),$$ which is absurd by Lemma \ref{Lemma:ModerationImplication}. This contradicts the assumption that $M_v$ is not a minimal complement of $X$ in $G_1\times G_2$. Hence the third statement follows. 

Note that the third statement implies the fourth statement. 

Assume that fourth statement holds. Then some $(u, \varphi)$-bounded spiked subset $X$ of $\calG$ with respect to $G_1, G_2$, having $\calB$ as its base, admits a minimal complement in $G_1\times G_2$ which is the graph of any $\varphi$-moderation $v$ of $u$ restricted to some subset $M'$ of $G_1$. By Lemma \ref{Lemma:ModerationImplication}, it follows that 
$$\calB\times G_2 + \{(x, v(x))\,|\, x\in M'\} = G_1\times G_2.$$ Hence 
$$\calB + M'=G_1.$$ So $M'$ is a complement to $\calB$ in $G_1$. If $M'$ were not a minimal complement to $\calB$ in $G_1$, then $\{(x, v(x))\,|\, x\in M'\}$ would not be a minimal complement to $X$ in $G_1\times G_2$, which would be absurd. So $M'$ is a minimal complement to $\calB$ in $G_1$. This proves the second statement, which implies the first statement by Theorem \ref{Thm:MinCompInSubgp}.
\end{proof}

\begin{remark}
Note that the image of the minimal complement $M_v$ of $X$ (as in the statement of Theorem \ref{Thm:LinesNoise}) under the projection map $$\bbZ^{k+1} = \bbZ^k \times \bbZ \to \bbZ^k$$ is a minimal complement of $\calB$ (since it is equal to $M$). However, it is not true that the image of any minimal complement of $X$ under the projection map $\bbZ^{k+1} = \bbZ^k \times \bbZ \to \bbZ^k$ is a minimal complement of $\calB$. 
For instance, consider the minimal complement $$\{(\pm 2n, 2n)\,|\, n\in \bbZ_{\geq 1}\}$$ of the bounded spiked subset $$\left((2\bbZ+1)\times \bbZ_{\leq 0} \right) \cup \left(\{0\}\times \bbZ\right)$$ of $\bbZ^2$. 
Its image under the projection map onto the first coordinate is equal to $2\bbZ$, which is not a complement of $\{0\}$. 
This example also indicates that it is not true in general that a $u$-bounded spiked set has the graph of a moderation of $u$ restricted to the minimal complement of its base as its only minimal complement. 
\end{remark}

\begin{corollary}
\label{Cor:LowerWhateverYAxisRota}
Let $f: \bbZ \to \bbR$ be a function and let $W_+$ (resp.$W_-$) denote the subset of $\bbZ^2$ lying `above' (resp. `below') the graph of $X(Y-f(X))=0$ (i.e. the union of the $y$-axis and the graph of $Y=f(X)$) rotated clockwise by an angle $\theta$, i.e., 
\begin{align*}
	W_+ & = 
	\left\{
	(x, y)\in \bbZ^2 \,| \, x\sin \theta + y \cos \theta > f(x\cos \theta - y \sin\theta)
	\right\}
	\cup 
	\{
	(x,y)\in \bbZ^2 \,|\, y = x \tan \theta
	\}	,\\
	W_- & = 
	\left\{
	(x, y)\in \bbZ^2 \,| \,x\sin \theta + y \cos \theta < f(x\cos \theta - y \sin\theta)	\right\}
	\cup 
	\{
	(x,y)\in \bbZ^2 \,|\, y = x \tan \theta
	\}	.
\end{align*}
If $\tan \theta$ is rational\footnote{$\tan \frac \pi 2$ is to be interpreted as ``$\frac 10$' and hence as a rational.}, then each of the sets $W_\pm$ admits a minimal complement in $\bbZ^2$  \textnormal{[see Fig.\ref{Fig8}]}.
\end{corollary}

\begin{proof}
Suppose $a,b$ are integers with no common factor and $\tan \theta = \frac ab$. Choose integers $c,d$ such that $ad-bc = -1$. Let $G_1, G_2$ denote the subgroups $\{
	(x, y)\in \bbZ^2 \,| \,  c x = dy
	\}, \{
	(x, y)\in \bbZ^2 \,| \,  a x = by
	\}$ of $\bbZ^2$ respectively. 
Note that $\bbZ^2$ is equal to $G_1G_2$, which we identify with $G_1\times G_2$ via the product map. 
Let $\varphi, u,u'$ denote the maps as defined in the proof of Corollary \ref{Cor:LowerWhateverRota}.
Then it follows that 
\begin{align*}
W_+ & = (\{0\} \times G_2)\bigcup \left( \cup_{g\in G_1} \{g\}\times \varphi^\mo \left(\bbZ_{<\varphi(u(g))}\right) \right),\\
W_- & =  (\{0\} \times G_2)\bigcup \left( \cup_{g\in G_1} \{g\}\times \varphi^\mo \left(\bbZ_{>\varphi(u'(g))}\right) \right).
\end{align*}
By Theorem \ref{Thm:NecessarySuffi}, the corollary follows.
\end{proof}

\begin{corollary}
\label{Prop:LinesNoisek1k2}
Let $k_1, k_2$ be positive integers and $u:\bbZ^{k_1}\to \bbZ^{k_2}$ be a function. It admits a moderation $v: \bbZ^{k_1}\to \bbZ^{k_2}$. Let $\calB$ be a subset of $\bbZ^{k_1}$ admitting a minimal complement $M$ in $\bbZ^{k_1}$. Then the subset 
$$M_v:=\{(x, v(x))\,|\, x\in M\}$$
of $\bbZ^{k_1+k_2}$ is a minimal complement of any subset $X$ of $\bbZ^{k_1+k_2}$ satisfying 
$$\calB\times \bbZ^{k_2}
\subseteq
X
\subseteq 
(\calB\times \bbZ^{k_2} )
\bigsqcup 
\left(
\bigsqcup 
_{x\in \bbZ^{k_1}\setminus \calB}
\left(
\{
x
\}
\times 
\bbZ^{k_2}_{<u(x)}
\right)
\right).
$$
Moreover, for each finite index subgroup $G$ of $\bbZ^{k_2}$, the function $u$ has a moderation $v_G:\bbZ^{k_1} \to \bbZ^{k_2}$ which takes values in $G$, and any subset $Y$ of $\bbZ^{k_1+k_2}$ satisfying 
$$\calB\times G
\subseteq
Y
\subseteq 
(\calB\times G )
\bigsqcup 
\left(
\bigsqcup 
_{x\in \bbZ^{k_1}\setminus \calB}
\left(
\{
x
\}
\times 
\bbZ^{k_2}_{<u(x)}
\right)
\right)
$$
has the subset $M_{v_G}C$ of $\bbZ^{k_1+k_2}$ as its minimal complement in $\bbZ^{k_1+k_2}$ 
where 
$$M_{v_G}:=\{(x, v_G(x))\,|\, x\in M\},$$
and $C$ is a set of coset representatives of $G$ in $\bbZ^{k_2}$.
\end{corollary}

\begin{proof}
The existence of a moderation $v$ of $u$ follows from Proposition \ref{Prop:ModerationExists}. The next statement follows from Theorem \ref{Thm:NecessarySuffi}. The existence of a $G$-valued moderation of $u$ for  each finite index subgroup $G$ of $G_2$ follows from Proposition \ref{Prop:ModerationExists}. 
From Theorem \ref{Thm:NecessarySuffi}, it follows that $M_{v_G}$ is a minimal complement of $Y$ in $\bbZ^{k_1}\times G$. Then from the proof of Theorem \ref{Thm:MinCompInSubgp}, we conclude that $M_{v_G}C$ is a minimal complement of $Y$ in $\bbZ^{k_1+k_2}$. 
\end{proof}

\begin{figure}[h]
	\centering
	\begin{subfigure}{.5\textwidth}
		\centering
		\includegraphics[scale=0.08]{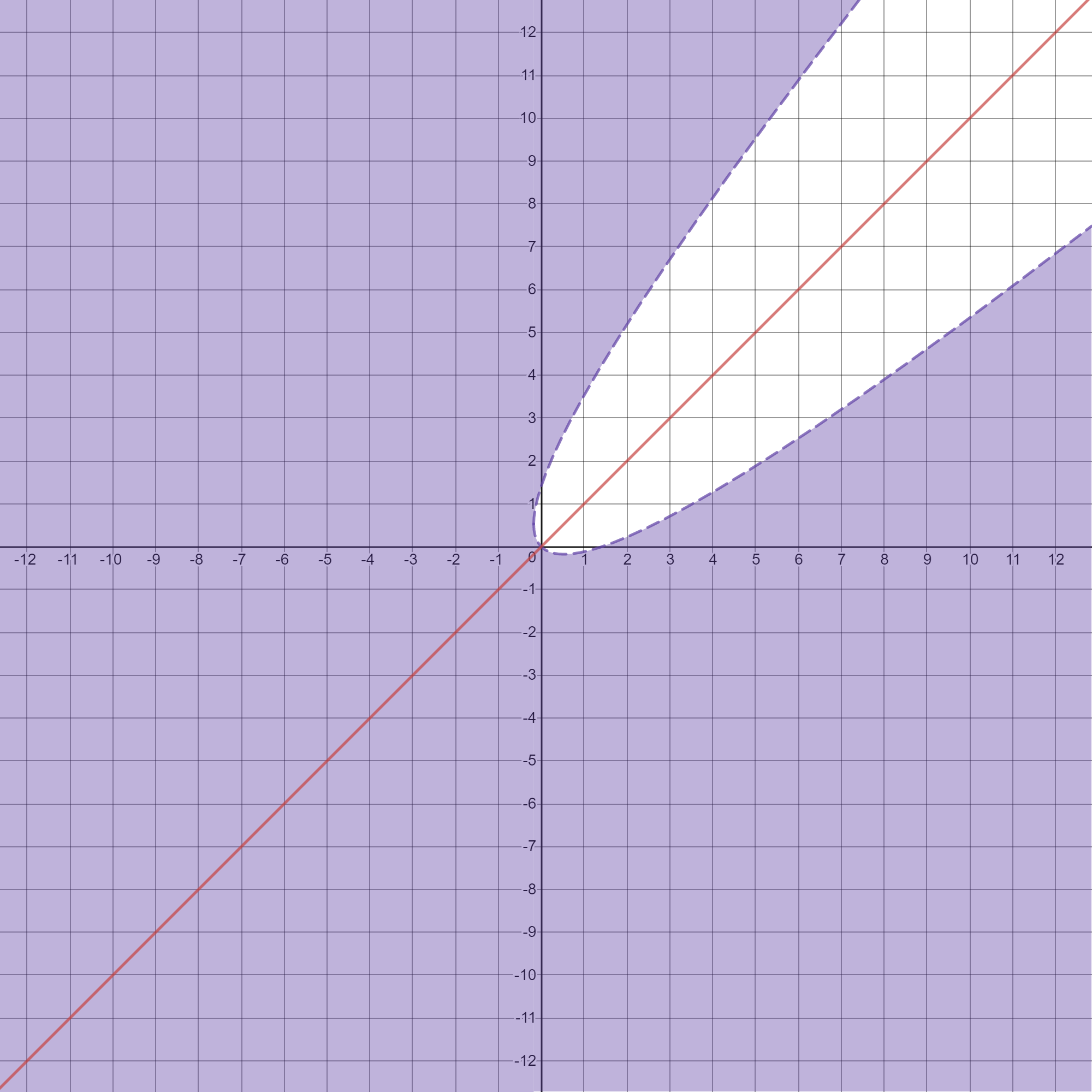}
		\caption{$W$}
		\label{}
	\end{subfigure}%
	\begin{subfigure}{.5\textwidth}
		\centering
		\includegraphics[scale=0.08]{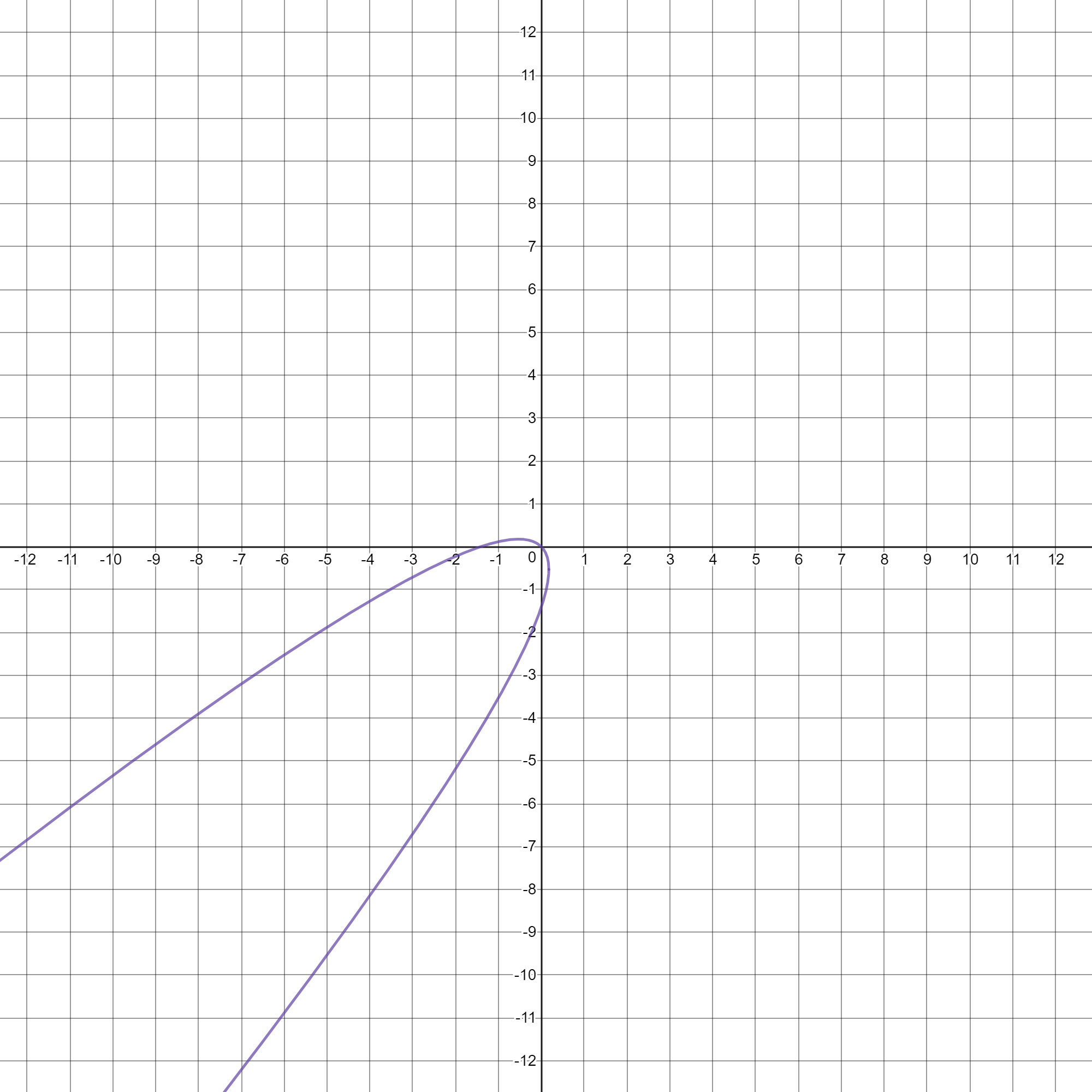}
		\caption{$M$}
		\label{}
	\end{subfigure}
	\caption{Cor. \ref{Cor:LowerWhateverYAxisRota}. A (slight) perturbation of $M$ is a minimal complement to $W = \lbrace (x,y)\in \mathbb{Z}^{2}| (x^{2} + y^{2} - 2xy -\sqrt{2}x-\sqrt{2}y) >0\rbrace \cup \lbrace (x,y)\in \mathbb{Z}^{2}| y=x\rbrace $ in $\mathbb{Z}^{2}$. Note that originally, $M=\lbrace (x,y)\in \mathbb{R}^{2}|(x^{2}+y^{2}-2xy + \sqrt{2}x + \sqrt{2}y)=0\rbrace$ only passes through $(0,0)$ in $\mathbb{Z}^{2}$.}
	\label{Fig8}
\end{figure}
\begin{example}\label{Cor5.9}
The set 
$$M=
\{
(s,t, -2(s^2+t^2), -(s^4 + t^4))\in \bbZ^4 \,| \, (s,t)\in \bbZ^2
\}$$
is a minimal complement of 
$$W = 
\{
(x, y,z, w)\in \bbZ^4 \,| \, z - x^2 - y^2<0, w - x^3 - y^3 <0
\}
\cup (\{(0,0)\}\times \bbZ^2)$$
in $\bbZ^4$.
\end{example}

\begin{remark}
If $u: \bbZ^{k_1} \to \bbZ^{k_2}$ is defined by $k_2$-polynomials over $\bbZ$ in $k_1$-variables, then from Remark \ref{Rk:PolyModeration}, it follows that $v$ can also be taken to be a function $\bbZ^{k_1}\to \bbZ^{k_2}$ defined by $k_2$-polynomials over $\bbZ$ in $k_1$ variables. 
\end{remark}

In fact, a larger class of sets admit minimal complements:

\begin{theorem}
	\label{Thm5.11}
	Let $G_1, G_2$ be subgroups of an abelian group $\calG$ such that $G_2$ is free of positive rank, $G_1\cap G_2$ is trivial. Let $\calB$ be a subset of $G_1$ with minimal complement $M$ in $G_1$. Let $u: G_1\to G_2$ be a function and 
	$$\varphi: G_2 \xra{\sim} \bbZ^{\rk G_2}$$ 
	be an isomorphism. Let $G_2'$ be a subgroup of $G_2$ of finite index and $g_2$ be an element of $G_2$. Let $X$ be a subset of $\calG$ containing $\calB G_2$ and contained in 
	$$
	\calB G_2
	\bigsqcup 
	\left(
	\bigsqcup_{g_1\in G_1\setminus \calB}
	g_1\cdot
	\left(
	\varphi^\mo \left(\bbZ^{\rk G_2}_{<\varphi(u(g_1))}
	\right)
	\right)
	\right)
	\bigsqcup 
	\left(
	\bigsqcup_{g_1\in G_1\setminus \calB}
	g_1\cdot
	\left(
	\left(
	\varphi^\mo \left(\bbZ^{\rk G_2}_{\not<\varphi(u(g_1))}
	\right)
	\right)
	\cap 
	\left(
	G_2\setminus g_2 G_2'
	\right)
	\right)
	\right).
	$$
	Then $u$ admits a $G_2'$-valued $\varphi$-moderation and the graph $M_v$ of the restriction $v'|_M$ of such a moderation $v'$ to $M$, i.e.,
	$$
	M_{v'}
	=
	\{(x,v'(x))\,|\, x\in M\}
	$$
	is a minimal complement of $X$ in $G_1\times G_2$. 
\end{theorem}

\begin{proof}
	Since $G_2'$ is a finite index subgroup of $G_2$, the existence of a $G_2'$-valued $\varphi$-moderation $v'$ of $u$ follows from Proposition \ref{Prop:ModerationExists}.
	
	Note that $M_{v'}$ is a complement of $\calB G_2$ in $G_1\times G_2$. So it is a complement of $X$ in $G_1\times G_2$. Suppose it is not a minimal complement of $X$ in $G_1\times G_2$. Then for some $g_1\in M$, the set $M_{v'}\setminus \{ (g_1, v'(g_1))\}$ is also a complement of $X$ in $G_1\times G_2$. Since $M$ is a minimal complement of $\calB$ in $G_1$, there exists an element $g_0\in G_1$ which does not belong to $\calB + (M\setminus \{g_1\})$. Since $v'$ is a $\varphi$-moderation of $u$, there exists an element $m_0\in \bbZ^{\rk G_2}$ such that the function $x\mapsto \varphi(v'(x) + u(x_0-x))$ defined on $G_1$ is bounded by $m_0$. Let $g'$ be an element of $g_2G_2'$ larger than $m_0$. From the proof of Lemma \ref{Lemma:ModerationImplication}, it follows that 
	$$
	\left(
	\{(x, v(x))\,|\, x\in G_1\} + 
	\left(
	\bigsqcup_{g_1\in G_1}
	g_1\cdot
	\left(
	\varphi^\mo \left(\bbZ^{\rk G_2}_{<\varphi(u(g_1))}
	\right)
	\right)
	\right) 
	\right)
	\cap \{g_0\} \times G_2
	$$
	is contained in $\{g_0\}\times \varphi^\mo \left(\bbZ^{\rk G_2}_{<\varphi(u(g_1))}
	\right)$ and hence does not contain $g_0 g'$.
	Also note that  
	$$
	(G_1\times G_2')+ 
	\left(
	\bigsqcup_{g_1\in G_1\setminus \calB}
	g_1\cdot
	\left(
	\left(
	\varphi^\mo \left(\bbZ^{\rk G_2}_{\not<\varphi(u(g_1))}
	\right)
	\right)
	\cap 
	\left(
	G_2\setminus g_2 G_2'
	\right)
	\right)
	\right)
	$$ 
	is contained in 
	$G_1\times (G_2\setminus g_2G_2')$, and hence does not contain $g_0g'$. Moreover, $g_0g'$ does not belong to $(M_{v'}\setminus \{(g_1, v'(g_1))\}) + \calB G_2$. Consequently, $g_0g'$ does not belong to $(M_{v'}\setminus \{(g_1, v'(g_1))\}) + X$, which is absurd. Hence $M_{v'}$ is a minimal complement of $X$ in $G_1\times G_2$. 
\end{proof}

\section{Concluding remarks}
We construct several new examples of sets in higher rank abelian groups having minimal complements and in each case, specify one of them.

\begin{example}
	The set 
	$$M=
	\{
	(t, -3t^2)\in \bbZ^2 \,| \, t\in \bbZ
	\}$$
	is a minimal complement of 
	$$\{
	(x, y)\in \bbZ^2 \,| \, y - x^2 <0
	\}
	\cup (\{0\}\times \bbZ)
	\cup 
	\{
	(m, 3^{|m|}n )\,|\, (m,n)\in \bbZ^2
	\}$$
	in $\bbZ^2$.
\end{example}

\begin{example}
	Let $\{p_k\}_{k\geq 1}$ be a sequence of odd primes. Then the set 
	$$M=
	\{
	(t, -2t^2)\in \bbZ^2 \,| \, t\in \bbZ
	\}$$
	is a minimal complement of the set 
	$$\{
	(x, y)\in \bbZ^2 \,| \, y - x^2 <0
	\}
	\cup (\{0\}\times \bbZ)
	\cup 
	\{
	(m, p_{|m|} n )\,|\, (m,n)\in (\bbZ\setminus \{0\})\times \bbZ
	\}
	$$
	in $\bbZ^2$.
\end{example}

\begin{example}
	The set 
	$$M=
	\{
	(t, -3t^2, -4t^4)\in \bbZ^3 \,| \, t\in \bbZ
	\}$$
	is a minimal complement of 
	$$\{
	(x, y,z)\in \bbZ^3 \,| \, y - x^2 <0, z - x^3 <0
	\}
	\cup (\{0\}\times \bbZ^2)
	\cup 
	\{
	(i, 3^{|i|}j ,4^{|i|}k )\,|\, (i,j,k)\in \bbZ^2
	\}$$
	in $\bbZ^2$.
\end{example}

\section{Acknowledgements}
The first author would like to acknowledge the fellowship of the Erwin Schr\"odinger International Institute for Mathematics and Physics (ESI) and would also like to thank the Fakult\"at f\"ur Mathematik, Universit\"at Wien where a part of the work was carried out.
The second author would like to acknowledge the Initiation Grant from the Indian Institute of Science Education and Research Bhopal, and the INSPIRE Faculty Award from the Department of Science and Technology, Government of India. The figures were made using Desmos.
\def\cprime{$'$} \def\Dbar{\leavevmode\lower.6ex\hbox to 0pt{\hskip-.23ex
  \accent"16\hss}D} \def\cfac#1{\ifmmode\setbox7\hbox{$\accent"5E#1$}\else
  \setbox7\hbox{\accent"5E#1}\penalty 10000\relax\fi\raise 1\ht7
  \hbox{\lower1.15ex\hbox to 1\wd7{\hss\accent"13\hss}}\penalty 10000
  \hskip-1\wd7\penalty 10000\box7}
  \def\cftil#1{\ifmmode\setbox7\hbox{$\accent"5E#1$}\else
  \setbox7\hbox{\accent"5E#1}\penalty 10000\relax\fi\raise 1\ht7
  \hbox{\lower1.15ex\hbox to 1\wd7{\hss\accent"7E\hss}}\penalty 10000
  \hskip-1\wd7\penalty 10000\box7}
  \def\polhk#1{\setbox0=\hbox{#1}{\ooalign{\hidewidth
  \lower1.5ex\hbox{`}\hidewidth\crcr\unhbox0}}}
\providecommand{\bysame}{\leavevmode\hbox to3em{\hrulefill}\thinspace}
\providecommand{\MR}{\relax\ifhmode\unskip\space\fi MR }
\providecommand{\MRhref}[2]{%
  \href{http://www.ams.org/mathscinet-getitem?mr=#1}{#2}
}
\providecommand{\href}[2]{#2}

\end{document}